\documentclass[aop]{imsart}

\usepackage{amsmath,amsthm,amssymb,mathrsfs,dsfont,hyperref, mathtools, bbm}
\RequirePackage[numbers]{natbib}

\startlocaldefs
\newcommand{\R}{\mathbb{R}}
\newcommand{\E}{\mathbb{E}}
\renewcommand{\P}{\mathbb{P}}
\newcommand{\e}{\mathbb{E}}

\newcommand{\tb}{2\beta}

\newcommand{\la}{\langle}
\newcommand{\ra}{\rangle}

\newtheorem{theorem}{Theorem}[section]

\newtheorem{lemma}[theorem]{Lemma}
\newtheorem{prop}[theorem]{Proposition}

\newtheorem{remark}[theorem]{Remark}

\newtheorem{conjecture}[theorem]{Conjecture}

\endlocaldefs

\begin{document}

\begin{frontmatter}
\title{Free energy of a diluted spin glass model \\
	with quadratic Hamiltonian}
%\title{A sample article title with some additional note\thanksref{T1}}
\runtitle{Free energy of a diluted spin glass model
	with quadratic Hamiltonian}
%\thankstext{T1}{A sample of additional note to the title.}

\begin{aug}
%%%%%%%%%%%%%%%%%%%%%%%%%%%%%%%%%%%%%%%%%%%%%%%
%% Only one address is permitted per author. %%
%% Only division, organization and e-mail is %%
%% included in the address.                  %%
%% Additional information can be included in %%
\author[A]{\fnms{Ratul} \snm{Biswas}\ead[label=e1]{biswa087@umn.edu}},
\author[A]{\fnms{Wei-Kuo} \snm{Chen}\ead[label=e2,mark]{wkchen@umn.edu}}
\and
\author[A]{\fnms{Arnab} \snm{Sen}\ead[label=e3,mark]{arnab@umn.edu}}
%%%%%%%%%%%%%%%%%%%%%%%%%%%%%%%%%%%%%%%%%%%%%%
%% Addresses                                %%
%%%%%%%%%%%%%%%%%%%%%%%%%%%%%%%%%%%%%%%%%%%%%%
\address[A]{University of Minnesota, \printead{e1,e2,e3}}

%\address[B]{???, \printead{e2,e3}}
\end{aug}

\begin{abstract}
We study a diluted mean-field spin glass model with a quadratic Hamiltonian. Our main result establishes the limiting free energy in terms of an integral of a family of random variables that are the weak limits of the quenched variances of the spins in the system with varying edge connectivity. The key ingredient in our argument is played by the identification of these random variables as the unique solution to a recursive distributional equation. Our results in particular provide the first example of the diluted  Shcherbina-Tirozzi model, whose limiting free energy can be derived at any inverse temperature and external field.  
\end{abstract}



\end{frontmatter}
%%%%%%%%%%%%%%%%%%%%%%%%%%%%%%%%%%%%%%%%%%%%%%
%% Please use \tableofcontents for articles %%
%% with 50 pages and more                   %%
%%%%%%%%%%%%%%%%%%%%%%%%%%%%%%%%%%%%%%%%%%%%%%
%\tableofcontents

%%%%%%%%%%%%%%%%%%%%%%%%%%%%%%%%%%%%%%%%%%%%%%
%%%% Main text entry area:

\section{Introduction}
In the study of fully connected mean-field spin glass models,  intensive investigations on the Sherrington-Kirkpatrick (SK) model as well as their variants have attracted a lot of attention and resulted in many rigorous mathematical results justifying physicists' observations and theories for the past decades, see \cite{MPV87,Pan13,Tal111,Tal112}. While the strength of the spin interactions in the SK model is uniformly defined across all sites, a set of more realistic mean-field spin glass models, called the diluted spin glasses, was introduced to capture the situtation in which the spins are only allowed to interact with a bounded number of neighbors on average, see, for instance, \cite{KS87,MP01,VB85}. In particular, within the Parisi framework of replica symmetry breaking theory,  the work of M\'ezard and Parisi \cite{MP01} proposed an ansatz to understand a spin glass model in the Bethe lattice. Later several attempts to justify the ansatz of \cite{MP01}  have been conducted in the setting of the diluted $p$-spin model and the diluted random $K$-SAT model, see \cite{coja2019spin,franz2003replica,Pan13Diluted,Pan14Diluted3,Pan15Diluted2,PT04}. Studies of high temperature behavior of diluted models have also  been performed in the SK model \cite{kosters2006fluctuations}, the $V$-statistics model \cite{talagrand2016mean}, the random $K$-SAT model \cite{monasson1997statistical,Pan14Diluted4,Tal01KSAT}, and the Viana-Bray model \cite{guerra2004high}.

Besides the studies above, a natural class of mean-field models that is of great importance and interest is the Shcherbina-Tirozzi (ST) model \cite{shcherbina2002central,shcherbina2003rigorous} which originated from the study of Gardner's problem \cite{gardner1987max,gardner1988space}. See \cite{bolthausen2021gardner,stojnic2013another} for more recent studies of Gardner's problem. In addition to being defined through a full rank disorder matrix that captures the mean-field interactions among all sites as in the SK model, the Hamiltonian in the ST model is set to be convex in the spin configuration space so that the corresponding Gibbs measure is strictly log-concave. Due to this feature, the usual concentration inequality for log-concave measures, such as the Brascamp-Lieb inequality, readily implies that any Lipschitz observable is concentrated around its Gibbs average. Consequently, it is expected that the solution to the ST model should always be replica symmetric in the sense that the Gibbs expectations of the spins are asymptotically independent and distributed as a function of i.i.d.~Gaussian random variables, as an outcome of central limit theorem, parametrized by a set of order parameters that are the solution to a system of consistency equations. If the latter is known to possess a unique solution, the replica symmetric solution can be justified and further used to compute the limiting free energy, see \cite{barbier2021performance,shcherbina2002central,shcherbina2003rigorous,Tal111}. However, while this is usually the case when the model is at very high temperature, the completely solvable cases in the entire temperature regime are only known in the Gardner model as well as in the quadratic model, see \cite{barbier2021performance,shcherbina2003rigorous}.

The diluted version of the ST model was proposed in Talagrand's book \cite[Section 6.7]{Tal111}, where the spin interaction is described instead by a sparse random matrix. As in the original ST model, one again expects the diluted ST model to exhibit replica symmetric solutions at any temperature, but this mechanism becomes more subtle, due to the sparsity of the disorder matrix, in which the Gibbs expectations of the spins should be governed by certain probability distribution that is the fixed-point solution to some distributional equation. Under the setting that the Hamiltonian is Lipschitz, the rigorous proof of this description has been carried out in \cite[Section 6.7]{Tal111}, but there remains a missing key ingredient, that is, the justification of the uniqueness of the fixed-point solution to the underlying distributional equation at any temperature, see \cite[Research Problem 6.7.14]{Tal111}. If this is achieved, then one can express the limiting free energy of the model in terms of this solution via the Aizenmann-Sims-Starr scheme \cite{ASS03}. Incidentally, such uniqueness can be achieved by a contraction argument at very high temperature, see \cite[Theorem 6.6.1]{Tal111}. However, validating this behavior throughout the entire temperature regime is a challenging problem in any reasonable setting.

In the present paper, we propose to study a diluted ST model with a quadratic Hamiltonian disordered by a symmetric distribution with finite second moment. Under our setting, the Gibbs measure readily reads as a high-dimensional Gaussian measure with covariance matrix $A_N^{-1}$, where $A_N$ is an $N\times N$ sparse random matrix written as the sum of the identity matrix along with a Poisson number of rank-one matrices, where the edges are formed only among $p$ uniformly chosen vertices. Furthermore, the free energy can be written as a function of the matrix $A_N$. Our study achieves two main results. First of all, we show that the joint spin variances, $(A_{N}^{-1})_{11},\ldots,(A_N^{-1})_{nn},$ are asymptotically independent and identically distributed. More importantly, we identify the weak limit of these objects as the unique fixed-point solution to a distributional equation. Based on this, the second main result settles a neat expression for the limiting free energy by a cavity computation. 

We emphasize that although our study specializes to the quadratic Hamiltonian, it turns out that the analysis is highly dedicated to reflecting the nature of dilution of the model in spite of being replica symmetric. To the best of our knowledge, this is the first example of the diluted ST model for which the limiting free energy can be completely solved in the entire temperature regime and at any external field. We believe some of the ideas in our approach might be potentially useful in understanding related diluted ST models. Additionally, the presence of the matrix $A_N$ can be viewed as a model for covariance matrix constructed from $p$ sparse vectors. It is an interesting random matrix model on its own and is worth investigating.

\subsection{The model and main results}

Consider a positive real $\alpha$ and a natural number $p$. Let $\beta>0$ and $h\in \mathbb{R}$ be the (inverse) temperature and external field parameters, respectively. For any $N\geq 1$, define the quadratic Hamiltonian by
\begin{align*}
	-H_N(\sigma) = -\beta \sum_{k=1}^M \Bigl( \sum_{r = 1}^p g_{k,I(k,r)} \sigma_{I(k,r)}\Bigr)^2 + h\sum_{i=1}^N \sigma_i
\end{align*}
for $\sigma\in \mathbb{R}^N,$
where $M$ is a Poisson random variable with mean $\alpha N$, the disorder matrix $g = (g_{k,i})_{k \geq 1, 1\leq i \leq N}$ consists of i.i.d.\ entries sampled from a symmetric distribution $\mathcal{D}$ with finite second moment, and 
\begin{align}\label{add:eq22}
	\mathcal{I}=\big( I(k,1),\ldots, I(k,p) \big ) _{k\geq 1}
\end{align} is a collection of  i.i.d. random vectors sampled uniformly from the index set $\{(i_1,\ldots,i_p):1\leq i_1,\ldots,i_p\leq N\,\,\mbox{are distinct}\}.$ All randomness here are independent of each other. Denote by $\eta_N$ the $N$-dimensional standard Gaussian measure on $\mathbb{R}^N.$ Define the free energy and the Gibbs measure associated to $H_N$, respectively, by
\begin{align*}
	F_N=\frac{1}{N}\log \int e^{-H_N(\sigma)}\eta_N(d\sigma)
\end{align*}
and
\begin{align*}
	G_N(d\sigma)&=\frac{e^{-H_N(\sigma)}\eta_N(d\sigma)}{\int e^{-H_N(\sigma')}\eta_N(d\sigma')}.
\end{align*}
For i.i.d.\ samples (or replicas) $\sigma,\sigma^1,\sigma^2,\ldots$ from $G_N$, we use $\la \cdot\ra$ to denote the Gibbs expectation with respect to these random variables. 
In view of the quadratic Hamiltonian $H_N$, the Gibbs measure is a Gaussian measure in $\mathbb{R}^N$ (explained in detail in Section \ref{sec:concentration}) with mean $\mu_N=hA_N^{-1}\mathbbm{1}$ (where $\mathbbm{1} = (1,\ldots, 1) \in \R^N$) and covariance matrix $A_N^{-1}$, where $A_N$ is a sparse $N\times N$ matrix defined  as \begin{align}\label{def_A}
	A_N = I_N+ 2\beta \sum_{k=1}^M v_kv_k^T.
\end{align}
Here, $I_N$ is an $N\times N$ identity matrix and $v_k$ is the vector, whose nonzero entries are  $g_{k, I(k,r)}$ at the position $I(k,r)$ for $1 \le r \le p$. The free energy can  be computed explicitly in terms of $\mu_N$ and $A_N,$
\begin{align}\label{add:eq1}
	F_N&=\frac{h}{2}\frac{\sum_{i=1}^N(\mu_N)_i}{N}+\frac{\log \det A_N}{2N}=\frac{h^2}{2}\frac{\mathbbm{1}^TA_N^{-1}\mathbbm{1}}{N}+\frac{\log \det A_N}{2N}.
\end{align}

To prepare for the statements of our main results, we introduce an operator associated to our model. Let $\mathcal{P}([0,1])$ be the space of probability distributions on $[0,1]$.
%equipped with the Wasserstein $q$-distance for $q \geq 1$, namely, $$W_q(\mu_1,\mu_2)=(\inf \e|Y_1-Y_2|^q)^{1/q},$$
%where the infimum is taken with respect to all joint distributions $(Y_1,Y_2)$ satisfying that $Y_1\sim \mu_1$ and $Y_2\sim \mu_2.$ 
Let $T = T_\alpha$ be a self map on $\mathcal{P}([0,1])$ defined as
\begin{align}\label{fixedpointeq}
	T(\mu)=\mbox{Law of\,\,}\Bigl(1+\sum_{k=1}^{R}\frac{2\beta\zeta_k^2}{1+2\beta \sum_{r=1}^{p-1}X_{k,r}\xi_{k,r}^2}\Bigr)^{-1},
\end{align} 	
where $(\zeta_k)_{k\geq 1}$ and $(\xi_{k,r})_{k\geq 1, 1 \le r \leq p-1}$ are i.i.d.\ distributed according to the random variable $\mathcal{D}$ and $(X_{k,r})_{k\geq 1, 1 \le  r \leq p-1}$ are i.i.d.\ distributed according to $\mu\in\mathcal{P}([0,1]),$ $R$ is Poisson$(\alpha p),$ and these are all independent of each other. Our first main result says that $T$ is the key operator to describe the weak limit of the spin variances.

\begin{theorem}[Distribution and asymptotic independence of spin variances]
	\label{thm:fixedpoint}
	For any $\alpha>0$, $T_\alpha$ has a unique fixed point, $\mu(\alpha).$ Moreover, 	for any $n\geq 1,$ the vector of spin variances 
	\begin{align*}
		\bigl(\la (\sigma_1-\la\sigma_1\ra)^2\ra,\ldots,\la (\sigma_n-\la \sigma_n\ra)^2\ra\bigr)=\bigl((A_N^{-1})_{11},\ldots,(A_N^{-1})_{nn}\bigr)
	\end{align*}
	converges weakly to $(X_1,\ldots,X_n)$ as $N$ tends to infinity, where $X_1,\ldots,X_n$ are i.i.d.\ with a common distribution $\mu(\alpha).$ 
\end{theorem}

Now, using the fixed point distribution, we establish an expression for the limiting free energy. For $x \in (0,1]$, denote by $\mu(\alpha x)$ the unique solution to the distributional equation $T(\mu) = \mu$ in which $R $ is replaced with $\mbox{Poisson}(\alpha xp)$. 

\begin{theorem}[Free energy]\label{thm:freeenergy}
	Let $\alpha>0$ and $p\in \mathbb{N}.$ For any $\beta>0$ and $h\in \mathbb{R}$, we have that 	$$ F_N \xrightarrow[N\to\infty]{L_1} \frac{h^2}{2}\E X(1) + \frac{\alpha}{2} \int_0^1 \E \log \Bigl(1 + \tb \sum_{r=1}^p \zeta_r^2 X_r(x)\Bigr)dx,$$ where  for any $x\in (0,1],$ $X(x), X_r(x) \stackrel{i.i.d.}{\sim} \mu(\alpha x)$, $\zeta_r \stackrel{i.i.d.}{\sim}\mathcal{D}$, and they are independent of each other. 
\end{theorem}

The main feature in Theorem \ref{thm:freeenergy} is that although $F_N$ is written in terms of the mean and the covariance through \eqref{add:eq1}, the limiting free energy essentially depends  only on the spin variances of the system with varying edge connectivity. In addition, we mention that the formula here is obtained via a cavity argument in $M$, the number of edges. One can also perform a more conventional cavity computation in $N$, the number of sites, instead as in the Aizenmann-Sims-Starr scheme \cite{ASS03,Pan13Diluted}, but the resulting formula will become more complicated involving a difference of two major terms again in terms $\mu(\alpha x)$ for $x\in (0,1].$

\subsection{Open problems}

Following our study, there are a couple of problems yet to be understood. First, while our main result studies the limiting free energy, it is a relevant question to understand the asymptotic behavior of the Gibbs measure, i.e., the limiting distribution of $\sigma\sim N(\mu_N,A_N^{-1})$. From the convexity of $-H_N$, we expect that asymptotically the marginals of the spin configuration $\sigma$ are  i.i.d.\ Gaussians as $N$ tends to infinity. Thus, to figure out the limiting distribution of $\sigma$, we need to investigate the limiting joint distribution of $
(\la\sigma_1\ra,\la (\sigma_1-\la \sigma_1\ra)^2\ra)$, or equivalently, $(h\sum_{i=1}^N(A_N^{-1})_{1i},(A_N^{-1})_{11}).$
As we will discuss in Section \ref{proofsketch} on how $(A_N^{-1})_{11}$ converges to $X$ weakly via the Sherman-Morrison formula, it can also be checked (proof omitted in this paper) that if this pair converges weakly to some pair $(Y,X)$, then, when $p=2,$ the pair $(X,U)$ for $U:=Y/(hX)$ must satisfy the following consistency equation
\begin{align}\label{add:eq-4}
	(U,X)&\stackrel{d}{=}\Bigl( 1-\sum_{k=1}^{R}\frac{2\beta\zeta_k\xi_kX_k U_k}{1+2\beta\xi_k^2 X_k},\Bigl(1+\sum_{k=1}^{R}\frac{2\beta\zeta_k^2}{1+2\beta \xi_k^2X_{k}}\Bigr)^{-1}\Bigr),
\end{align} 	
where $(U_k,X_k)_{k\geq 1}$ are i.i.d.\ copies of $(U,X)$, $(\xi_k)_{k\geq 1}$ and $(\zeta_k)_{k\geq 1}$ are i.i.d.\ copies of $\mathcal{D},$ and $R$ is Poisson with mean $2\alpha.$ There are all assumed to be independent of each other. Theorem \ref{thm:fixedpoint} has shown that the second coordinate has a unique fixed point. We believe that the following is true.
\begin{conjecture} 
	\eqref{add:eq-4} has a unique fixed point.
\end{conjecture}

Throughout the entire paper, we assume that the distribution of the disorder random variable $\mathcal{D}$ is symmetric. As one shall see, Theorem \ref{thm:fixedpoint} holds in general without this assumption, but in the derivation of the limiting free energy, we require that $\e (A_N^{-1})_{ij}=0$ and $\e (A_N^{-1})_{ij}(A_N^{-1})_{ik}=0$ for distinct $i,j,k$ in order to deduce that $F_N\approx \e F_N$ and the right-hand side ultimately depends only on the main diagonal terms of $A_N^{-1}$ so that we can handle the limiting free energy via  the fixed point of $T$. Our argument to show that the mean and covariance of the off-diagonal terms $(A_N^{-1})_{ij}$ vanish uses the assumption that $\mathcal{D}$ is symmetric. From numerical simulations, it seems to indicate that these behaviors remain asymptotically true without the symmetry assumption. We thus close this subsection with

\begin{conjecture} 
	Theorem \ref{thm:freeenergy} holds for any distribution $\mathcal{D}$ with finite second moment.
\end{conjecture}

\subsection{Structure of the paper}

The rest of the paper is organized as follows. In Section 2, we provide an outline for the proofs of Theorems \ref{thm:fixedpoint} and \ref{thm:freeenergy}. In Section 3, we establish the concentration of the free energy. Section 4 contains the proof of the concentration of the generalized multi-overlap of a coupled system that will be used in  Section~5, in which we establish the asymptotic independence of spin variances. In Section 6, we prove the uniqueness of $\mu$ as being the solution to the distributional equation $T(\mu) =\mu$ and show how the limiting distribution of the spin variances gives rise to the distributional operator $T$ and  is distributed according to $\mu$. Gathering these results, we complete the proofs of Theorems~\ref{thm:fixedpoint} and \ref{thm:freeenergy} in Section 7. For technical purposes, our Sections 4 to 7 are handled assuming that $\mathcal{D}$ is bounded. In Section 8, we remove this restriction and generalize Theorems \ref{thm:fixedpoint} and \ref{thm:freeenergy} to the case that $\mathcal{D}$ has finite second moment.

\section{Proof sketch}\label{proofsketch}

To facilitate our proofs for Theorems \ref{thm:fixedpoint} and \ref{thm:freeenergy}, we describe our approach and outline some basic ideas in this section. 
As one shall see, our proof of Theorem \ref{thm:fixedpoint} is long and consists of a number of key steps, but it contains some essential ideas that might be potentially useful in related models and problems. 
\smallskip

\subsection{Proof sketch of Theorem \ref{thm:fixedpoint}}

$\,\,$

\smallskip

{\noindent \bf (i) Uniqueness of the fixed point:} We do not know how to show that the map $T$ is contractive. Our idea is to introduce a new operator $\mathcal{T}$ by conjugating $T$ with the logarithm map and show that $\mathcal{T}$ is contractive. For $q\geq 1,$ let $\mathcal{P}_q(\mathbb{R}_+)$ be the space of probability distributions on $\mathbb{R}_+=[0,\infty)$ with $\int x^q\nu(dx)<\infty.$ The map $\mathcal{T}$ from $\mathcal{P}_q(\mathbb{R}_+)$ to $\mathcal{P}_q(\mathbb{R}_+)$ is defined as
$$\mathcal{T} = \phi \circ T \circ \phi^{-1}$$ for $\phi(x) =  - \log x$, where $\phi^{-1}(\nu)$ and $\phi(\mu)$ are understood as the push-forward measures of $\mu\in \mathcal{P}([0,1])$ and $\nu\in \mathcal{P}_q(\mathbb{R}_+)$ under $\phi$ and $\phi^{-1}$, respectively. To show that $T(\mu)=\mu$ has a unique solution,  it is enough to show that $\mathcal{T}(\nu) = \nu$ has a unique solution in $\mathcal{P}_q(\mathbb{R}_+)$. To accomplish this, we equip $\mathcal{P}_q(\mathbb{R}_+)$ with  the Wasserstein $q$-distance and argue (see Lemma~\ref{add:lem9}) that $\mathcal{T}$ is indeed a contraction for sufficiently large $q \geq 1$ (depending on $\beta$).

\smallskip

\noindent {\bf (ii) Independence of spin variances:} The most technical step in our argument is to show that for any fixed $n\geq 1,$ the first $n$ terms $$\bigl((A_N^{-1})_{11},\ldots,(A_N^{-1})_{nn}\bigr)$$ in the main diagonal of $A_N^{-1}$ are asymptotically independent as $N$ tends to infinity. To prove this, we shall view $\tau:=\sigma-\la \sigma\ra$ as a centered  spin configuration sampled from a (centered) Gibbs measure $G_N^c$ corresponding to the Hamiltonian,
$$
X_N(\tau):=\beta \sum_{k \leq M} \Bigl( \sum_{r = 1}^p g_{k,I(k,r)} \tau_{I(k,r)}\Bigr)^2.
$$
In view of \eqref{def_A}, $\tau\sim N(0,A_N^{-1})$ and it suffices to show that $\la |\tau_1|^2\ra^c,\ldots,\la |\tau_n|^2\ra^c$ are asymptotically independent, where $\la \cdot\ra^c$ is the Gibbs expectation associated to the measure $G_N^c$. We adapt an analogous treatment from the study of the diluted ST model in  \cite[Section 6.7]{Tal111} by showing that the generalized overlaps are concentrated under the annealed measure $\e \la \cdot\ra^c$. More precisely, for any $\kappa\geq 1$ and any function $\phi:\mathbb{R}^\kappa\to \mathbb{R}$ of mild growth,
\begin{align}\label{add:eq-2}
	\lim_{N\to\infty}\e \bigl\la \bigl|Q-\e\bigl\la Q\bigr\ra^c\bigr|\bigr\ra^c=0
\end{align}
for $Q:=N^{-1}\sum_{i=1}^N\phi(\tau_i^1,\ldots,\tau_i^\kappa),$ where $\tau^1,\ldots,\tau^\kappa$ are i.i.d.\ replicas drawn from $G_N^c$.
Once this is valid, we utilize the symmetry among the replicas and spins to conclude that the entries in any weak limit of $(\la |\tau_1|^2\ra^c,\ldots,\la |\tau_n|^2\ra^c)$ must be independent. 
Although this deduction will be detailed in the proof of Proposition \ref{add:prop:indep}, to foster intuition, we provide a quick explanation on how \eqref{add:eq-2} implies that $\la|\tau_1|^2\ra^c$ and $\la|\tau_2|^2\ra^c$ are asymptotically uncorrelated. Let $\tau^1$ and $\tau^2$ be two replicas drawn from $G_N^c$. Consider $Q_1: = N^{-1}\sum_{i=1}^N |\tau_i^1|^2$ and $Q_2: = N^{-1}\sum_{i=1}^N |\tau_i^2|^2$. Then we can write
	\begin{align*}
\E \langle |\tau_1|^2\rangle^c \langle |\tau_2|^2\rangle^c & = \E \langle |\tau_1^1|^2|\tau_2^2|^2\rangle^c \\
&= \frac{1}{N(N-1)}\sum_{1\leq i \neq j \leq N} \E \langle |\tau_i^1|^2|\tau_j^2|^2\rangle^c \\
& \approx \frac{1}{N^2}\sum_{1\leq i, j \leq N} \E \langle |\tau_i^1|^2|\tau_j^2|^2\rangle^c = \E\langle Q_1Q_2\rangle^c \\
& \approx \E \langle Q_1\rangle^c\e \langle Q_2\rangle^c = \E \langle |\tau_1|^2\rangle^c\E \langle |\tau_2|^2\rangle^c,
\end{align*}
where the first approximation follows because the $N$ diagonal terms have a vanishing contribution, while the second approximation utilizes the concentrations of $Q_1$ and $Q_2$. Additionally, the second and last equalities use the symmetry of the spins in $i$.

The justification of \eqref{add:eq-2} relies on showing two concentrations, 
\begin{align}\label{add:eq-20}
	\lim_{N\to\infty}\e\bigl\la |Q-\la Q\ra^c|\bigr\ra^c=0
\end{align}
and 
\begin{align}
	\label{add:eq-21}
	\lim_{N\to\infty}\e \bigl|\la Q\ra^c-\e\la Q\ra^c\bigr|=0.
\end{align} To this end, we consider the coupled Hamiltonian, $$
X_N(\tau^1)+\cdots+X_N(\tau^\kappa)-\lambda NQ(\tau^1,\ldots,\tau^\kappa),$$ and the associated Gibbs expectation $\la \cdot\ra_\lambda^c$ and free energy $F_N^c(\lambda).$ First of all, from the convexity of the Hamiltonian $X_N$, it can be checked that the Gibbs measure associated to this coupled Hamiltonian is  strictly log-concave for small enough $\lambda$ and as a result, the Brascamp-Lieb inequality (see \cite{brascamp2002extensions} or \cite[Theorem 3.1.4]{Tal111}) readily implies that there exists a constant $K>0$ such that $\e\la |Q-\la Q\ra_\lambda^c|^2\ra_\lambda^c\leq K/N$ for any $|\lambda|$ small enough and $N\geq 1$. In particular, this implies \eqref{add:eq-20}. Next, to show \eqref{add:eq-21}, we use the fact that $\frac{d}{d\lambda}F_N^c(\lambda)=\la Q\ra_\lambda^c$ and the convexity of $F_N^c(\lambda)$ to bound
\begin{align*}
	\e\bigl\la\bigl|\la Q\ra^c-\e\la Q\ra^c\bigr|\bigr\ra^c&\leq \frac{1}{\lambda}\bigl(\e \bigl|F_N^c(\lambda)-\e F_N^c(\lambda)\bigr|+\e \bigl|F_N^c(-\lambda)-\e F_N^c(-\lambda)\bigr|\\
	&\qquad +\e \bigl|F_N^c(0)-\e F_N^c(0)\bigr|\bigr) + \e {F_N^c}'(\lambda)-\e {F_N^c}'(-\lambda).
\end{align*}
Here the last term can be controlled by $$\e {F_N^c}'(\lambda)-\e {F_N^c}'(-\lambda)=\int_{-\lambda}^\lambda \e{F_N^c}''(t)dt=N\int_{-\lambda}^\lambda \e\bigl\la \bigl(Q-\la Q\ra_t^c\bigr)^2\bigr\ra_{t}^cdt\leq 2K\lambda$$ after using the bound in the first part. Therefore, as long as we can show that there exists some constant $K'>0$ such that
\begin{align}\label{add:eq-19}
	\max_{|t|\leq\lambda }\e \bigl|F_N^c(t)-F_N^c(t)\bigr|\leq \frac{K'}{\sqrt{N}}
\end{align}
for all small enough $\lambda$ and $N\geq 1$, taking $\lambda=N^{-1/4}$ would then yield the second concentration since $\e\bigl\la\bigl|\la Q\ra-\e\la Q\ra^c\bigr|\bigr\ra^c\leq (3K'+2K)N^{-1/4}.$ We achieve \eqref{add:eq-19} by making use of the martingale difference technique. At the center of our argument, we will need a fourth-moment bound $\e \bigl[\la |\tau_i^\ell|^4\ra_\lambda^c|M\bigr]\leq K''(1+(M/N)^2)$.  The subtlety here is that this bound appears to be far from reach if $\lambda$ is fixed, but fortunately, it holds when $\lambda$ vanishes in the order $N^{-1/4}$ (see Lemma \ref{add:lem1}) and this is good enough to establish \eqref{add:eq-19}.

\begin{remark}\rm
	In \cite[Section 6.7]{Tal111}, a similar strategy for \eqref{add:eq-2} was also adapted to deal with the concentration of $Q$ in the diluted ST model, where the Hamiltonian is Lipschitz and this condition is strong enough to show that the exponential moment of the spins is bounded. In contrast, our Hamiltonian is quadratic and this results in extra complications in the analysis. For instance, controlling the fourth moment of $\tau_1$ becomes substantially involved, see Lemma \ref{add:lem1} below, which is the most crucial step in proving \eqref{add:eq-19}. 
	%A similar  fourth moment bound for the uncentered spin $\sigma_1$, which could have simplified subsequent analysis,  remains elusive. 
\end{remark}

{\noindent \bf (iii) Convergence of $(A_N^{-1})_{NN}$.} Next we continue to show that $(A_N^{-1})_{NN} $ converges weakly to $\mu$, where  $\mu$ is the unique solution of $T(\mu) = \mu$. Clearly,  $((A_N^{-1})_{NN})_{N \ge 1}$ is a tight family of random variables since $0 \le (A_N^{-1})_{NN} \le 1$. 
The key step is to show that for large $N$, $(A_N^{-1})_{NN} $ obeys the following approximate distributional recursion 
\begin{equation}\label{eq:dist_fx_pt1}
	(A_N^{-1})_{NN}  \stackrel{d}{\approx} T ((A_N^{-1})_{NN}).
\end{equation}
In the above, $T(X)$ denotes  the random variable with law $T(\mu_X)$ for $X \sim \mu_X$. If \eqref{eq:dist_fx_pt1} holds, then any limit point $Z$ of $(A_N^{-1})_{NN}$  would satisfy $Z \stackrel{d}{=} T(Z)$. So, by the uniqueness of the fixed point of $\mu = T(\mu)$, we must have $Z \sim \mu$. This guarantees the weak convergence of $(A_N^{-1})_{NN} $  to $\mu$. Once we prove that the marginals converge, the joint convergence of $((A_N^{-1})_{11},\ldots,(A_N^{-1})_{nn})$ easily follows from their asymptotic independence.

To establish \eqref{eq:dist_fx_pt1}, using the thinning property of the Poisson random variable, we  divide $v_1,\ldots,v_{M}$ into two groups $(u_k)_{k\leq Q}$  and $(w_k)_{k\leq R}$, where $Q\sim\mbox{Poisson}(\alpha(N-p))$, $R\sim\mbox{Poisson}(\alpha p )$, and they are independent. In the first group, $N$ does not appear in the index set $(I(k, 1), \ldots, I(k, p))$ of  $v_k$ and in the second,  $N$ appears as one of  indices $I(k, 1), \ldots, I(k, p)$.  Thus we can write
\begin{align*}
	A_N \stackrel{d}{=} I + \tb\sum_{k \leq Q} u_ku_k^T + \tb\sum_{k \leq R}w_kw_k^T = B_N +  \tb\sum_{k \leq R}w_kw_k^T,
\end{align*}
where the vector $w_k$ can be represented as 
\[ w_k =\sum_{r=1}^{p-1}\xi_{k,\bar I(k,r)} e_{\bar I(k,r)}+\zeta_k e_N,\]
where  $(e_i)_{1\leq i\leq N}$ is the standard basis, $(\zeta_k)_{k\geq 1}$ and $(\xi_{k,i})_{k\geq 1,1\leq i\leq N - 1}$ are i.i.d.\ sampled from $\mathcal{D}$, 
$(\bar I(k,1), \ldots, \bar I(k,p-1))_{k \geq 1}$ are i.i.d. uniformly sampled from $\{(i_1, \ldots, i_{p-1}): 1\leq i_r\leq N-1, \text{ all distinct}\}$,
and these are all independent of each other.

We view $A_N$ as a finite-rank perturbation of $B_N$ and  use the Woodbury matrix identity to write $(A_N^{-1})_{NN}$ in terms of entries of $B_N^{-1}$ and obtain that
\begin{equation}\label{approx1}
	( A_N^{-1})_{NN} \approx \Bigl(1 + \sum_{k=1}^R \frac{\tb\zeta_k^2}{1 + \tb\sum_{r=1}^{p-1}\xi_{k,\bar I(k,r)}^2 (B_N^{-1})_{\bar I(k,r),\bar I(k,r)}}\Bigr)^{-1}.
\end{equation}
In the above approximation, the error involves $O(R)$ entries of $B_N^{-1}$ of the form   (i) 
$(B_N^{-1})_{ \bar I(k, r),  \bar I(k, s)},$   $1 \le r \ne s \le p -1, 1 \le k \le R,$ and (ii) $(B_N^{-1})_{ \bar I(l, r),  \bar I(k, s)}, 1 \le r, s \le p -1, 1 \le k \ne l \le R$. Since the indices $\bar I(k, r)$ are independent of $B_N$, to control the approximation error we need to show that an off-diagonal entry of $B_N^{-1}$, chosen randomly independent of $B_N$,  is small. This can be argued easily based on the identity 
\[ (B_N^{-2})_{ii}  = \sum_{j=1}^N ((B_N^{-1})_{ij} )^2.\]
Taking expectation, using the fact that $0 \le  (B_N^{-1})_{ii},  (B_N^{-2})_{ii} \le  1$, and the symmetry of distribution of off-diagonal entries of $B_N^{-1}$, we obtain that
\[ 1 \ge (N-1) \e((B_N^{-1})_{12})^2,  \] 
yielding that $\e  | (B_N^{-1})_{12}|  \le (N-1)^{-1/2}$.

Now define 
\begin{align*}
	A'_N = B_N +  \tb\sum_{k \leq R'} (w'_k) (w'_k)^T,
\end{align*}
where $(R', (w'_k)_{k \ge 1})$ is an independent copy of $(R, (w_k)_{k \ge 1})$, which is also independent of $B_N$. Clearly, $A_N$ and $A'_N$ have the same distribution. From the resolvent identity and the fact that  random off-diagonal entries of $B_N^{-1}$ are small, we can derive that 
\begin{equation}\label{approx2}
	(B_N^{-1})_{\bar I(k,r),\bar I(k,r)} \approx   (A_N')^{-1}_{\bar I(k,r),\bar I(k,r)}  \stackrel{d}=   (A_N^{-1})_{\bar I(k,r),\bar I(k,r)}.
\end{equation}
Combining \eqref{approx1} and \eqref{approx2} and noting that  the diagonal entries $(A_N^{-1})_{\bar I(k,r),\bar I(k,r)}, 1 \le k \le R, 1 \le r \le p-1$  are asymptotically independent of each other and have the same distribution as $(A_N^{-1})_{NN}$, the equation \eqref{eq:dist_fx_pt1} follows as desired.

\subsection{Proof sketch of Theorem \ref{thm:freeenergy}}

First of all, we verify that $F_N$ is concentrated. Recall the formula  \eqref{add:eq1} of $F_N$. The second term can be written as the empirical spectral measure $\mu_{A_N-I}$ of $A_N - I$ as
\[ \frac{\log \det A_N}{2N}  = \frac{1}{2} \int \log (1+ \lambda) d\mu_{A_N - I}(\lambda),\]
whose concentration follows from an application of the rank inequality for the empirical spectral measure  and the Efron-Stein inequality. However, the first term is more delicate as it can not be related to any such empirical spectral measure. We establish the desired concentration by a key observation that the symmetry of the disorder distribution $\mathcal{D}$ implies that the off-diagonal entries of $A_N^{-1}$ are centered and are uncorrelated, see Lemma \ref{lem0}. Combining them together,  we arrive at
\begin{align*}
	%\label{freeenergy:eq1}
	F_N\approx \e F_N=\frac{h^2}{2N}\sum_{i=1}^N\e (A_{N}^{-1})_{ii}+\frac{\e\log \det A_N}{2N}=\frac{h^2}{2}\e (A_{N}^{-1})_{11}+\frac{\e\log \det A_N}{2N}.
\end{align*}
In the previous subsection, we have seen that the term $( A_{N}^{-1})_{11}$ converges weakly to $\mu(\alpha)$ and hence, so do their expectations. It remains to compute $\e\log \det A_N$. For this purpose, letting $S_l=I+2\beta \sum_{k=1}^l v_kv_k^T$ for $0\leq l \leq M$, we perform a cavity computation in $l$ by writing
\begin{align*}
	\frac{1}{N}\e\log \det A_N &=\frac{1}{N}\sum_{l=0}^{M-1}\log \frac{\det S_{l+1}} {\det S_{l}}=\e \frac{1}{N}\sum_{l=0}^{M-1}\log\bigl(1+2\beta v_{l+1}^T S_{l}^{-1}v_{l+1}\bigr)
\end{align*}
where the last equality used the matrix determinant lemma. We use $M/N \approx \alpha$ and notice that $v_{l+1}$ is independent of $S_l^{-1}$ to approximate the right-hand side as 
\begin{align*}
	\alpha  \e \frac{1}{M} \sum_{l=0}^{M-1}\log\bigl(1+2\beta v_{l+1}^TS_{l}^{-1}v_{l+1}\bigr) &\approx 
	\alpha  \e  \frac{1}{M + 1} \sum_{l=0}^{M}\log\bigl(1+2\beta v^TS_{l}^{-1}v \bigr)\\
	&=   \alpha  \e \log\bigl(1+2\beta v^TS_{L}^{-1}v \bigr),
\end{align*}
where $v$ is a copy of $v_1$ and $L$, conditionally on $M$,  is uniform from $\{0, 1, \ldots, M\}$ and independent of everything else. It follows from a property of Poisson distribution (see Lemma~\ref{add:lem10}) that $S_L$ has same distribution of as $A_N$, with the Poisson $(\alpha N)$ number of edges replaced by Poisson$(\alpha UN)$, where $U$ is an independent uniform random variable on $[0,1]$. As argued in the third part of the last subsection, given $U=x$,  $(S_{L}^{-1})_{ij}\approx 0$ for all $i\neq j$. So,  the quadratic form $v^TS_{L}^{-1}v $ essentially depends only on $p$ of the main diagonal terms in $S_{L}^{-1}$, which are asymptotically independent and equal to $\mu(\alpha x)$ in distribution if $U = x$. Averaging over $x$ gives rise to the integral formula in Theorem \ref{thm:freeenergy}.

\section{Concentration of the free energy}\label{sec:concentration}

The main result of this section establishes the concentration of the free energy.

\begin{theorem}\label{thm:concentration}
	There exists a positive constant $K$ independent of $N$ such that
	\begin{align*}
		\e\bigl|F_N-\e F_N\bigr|\leq \frac{K}{\sqrt{N}},\,\,\forall N\geq 1.
	\end{align*}
\end{theorem}

Before we turn to the proof of Theorem \ref{thm:concentration}, we explain why $G_N\sim N(\mu_N,A_N^{-1})$ and the validity of \eqref{add:eq1}, where $A_N$ is defined in \eqref{def_A} and $\mu_N=hA_N^{-1}\mathbbm{1}$. Henceforth, we drop the subscript $N$ in $A_N$ and $\mu_N$ for notational clarity. Since 
\begin{align*}
	-H_N(\sigma)-\frac{1}{2}\|\sigma\|^2&=-\frac{1}{2}\sigma^TA\sigma+h\sigma^T\mathbbm{1}=-\frac{1}{2}(\sigma-\mu)^TA(\sigma-\mu)+\frac{h^2}{2}\mathbbm{1}^TA^{-1}\mathbbm{1}
\end{align*}
for any bounded measurable function $f$, we can rewrite
\begin{align*}
	\int f(\sigma)e^{-H_N(\sigma)}\eta_N(d\sigma)&=\frac{e^{\frac{h^2}{2}\mathbbm{1}^TA^{-1}\mathbbm{1}}}{(2\pi)^{N/2}}\int f(\sigma)e^{-\frac{1}{2}(\sigma-\mu)^T(A^{-1})^{-1}(\sigma-\mu)}d\sigma
\end{align*}
and compute
\begin{align}\label{add:eq-5}
	\int  e^{-H_N(\sigma)}\eta_N(d\sigma)&=\frac{e^{\frac{h^2}{2}\mathbbm{1}^TA^{-1}\mathbbm{1}}}{(2\pi)^{N/2}}\int e^{-\frac{1}{2}(\sigma-\mu)^T( A^{-1})^{-1}(\sigma-\mu)}d\sigma=e^{\frac{h^2}{2}\mathbbm{1}^TA^{-1}\mathbbm{1}}\det(A)^{1/2}.
\end{align}
Consequently,
\begin{align*}
	%\label{gaussianmeasure}
	\la f(\sigma)\ra&=\frac{1}{(2\pi)^{N/2}\det(A)^{1/2}}\int f(\sigma)e^{-\frac{1}{2}(\sigma-\mu)^T(A^{-1})^{-1}(\sigma-\mu)}d\sigma.
\end{align*}
In other words, under $\langle \cdot \rangle$, $\sigma$ is a multi-normal random vector with mean $\mu=hA^{-1}\mathbbm{1}$ and covariance matrix $A^{-1}.$ From \eqref{add:eq-5}, we also see that
\begin{align}\label{add:eq7}
	F_N&=\frac{h^2}{2}\frac{\mathbbm{1}^TA^{-1}\mathbbm{1}}{N}+\frac{\log \det A}{2N}.
\end{align}
We establish the concentration of the free energy by showing that the two terms on the right-hand side are concentrated, which are based on the next two subsections.

\subsection{Mean and covariance of $A_{ij}^{-1}$}

The following lemma is one of the key ingredients in this paper, which studies the mean and covariance of the entries of $A_{ij}^{-1}$ and will be used to show the concentration of the first part in \eqref{add:eq7}.
\begin{lemma} \label{lem0} We have that
	\begin{align}
		\label{lem0:eq1}	\e A_{12}^{-1}&=0,\\
		\label{lem0:eq2}	\e A_{12}^{-1}A_{13}^{-1}&=0,\\
		\label{lem0:eq3}	\e A_{12}^{-1}A_{34}^{-1}&=0.
	\end{align}
\end{lemma}

We need a lemma before we turn to the proof of Lemma \ref{lem0}.

\begin{lemma}\label{lem3}
	For any $\ell\geq 0$ and $1\leq i\neq j\leq N,$ conditionally on $M$ and $$\{I(k,1), \ldots, I(k,p)\}_{1\leq k\leq M},$$ $A_{ij}^\ell$ can be written as a sum of the terms of the form $P(g)Q(g)$,
	where  $P$ is a polynomial and is even in each coordinate and
	\begin{align*}
		%\label{lem3:eq1}
		Q(g)&=\prod_{k=1}^M\prod_{b=1}^Ng_{k,b}^{d_k(b)}
	\end{align*}
	for  some nonnegative integers $d_k(b)$'s satisfying that $\sum_{k=1}^Md_k(b)$ is even for $b\neq i,j$ and
	$\sum_{k=1}^Md_k(i)$ and $\sum_{k=1}^Md_k(j)$ are both odd.
\end{lemma}

\begin{proof}
	Without loss of generality, we take $i=1$ and $j=2$. When $\ell = 0$, our assertion obviously holds since $A^{\ell}_{12} = 0$. Henceforth, we assume that $\ell \geq 1$.	
	For $1\leq k\leq M$, let $\mathcal{I}(k): = \{I(k,1), \ldots, I(k,p)\}$. From the definition of $A$, for any indices $i$ and $j$, we have \begin{align*}
	A_{ij} & = \delta_{ij} + \tb\sum_{k:\; i,j\in \mathcal{I}(k)}g_{k,i}g_{k,j}.
	\end{align*}
	Following the convention that $i_0=1$ and $i_\ell = 2$, we write \begin{align*}
	A^{\ell}_{12} = \sum_{i_1, \ldots, i_{\ell-1}}A_{1i_1}A_{i_1i_2}\ldots A_{i_{\ell-1}2} = \sum_{i_1, \ldots, i_{\ell-1}}A_{i_0i_1}A_{i_1i_2}\ldots A_{i_{\ell-1}i_\ell}.
	\end{align*}
	Fix a choice of the indices $i_1, \ldots, i_{\ell-1}$ for the remainder of the proof. Define $S: = \{0\leq j\leq \ell-1: i_j = i_{j+1}\}$. Then we can write \begin{align*}
	A_{i_0i_1}A_{i_1i_2}\ldots A_{i_{\ell-1}i_\ell} = \prod_{j\in S}A_{i_ji_j} \cdot \prod_{j \in S^C} A_{i_ji_{j+1}}.
	\end{align*}
	The function \begin{align*}
	P(g): = \prod_{j\in S}A_{i_ji_j} = \prod_{j\in S}\Bigl(1 + \tb\sum_{k: i_j \in \mathcal{I}(k)} g_{k,i_j}^2\Bigr)
	\end{align*}
	is a polynomial which is even in each coordinate. To describe the rest, we note that each term in the expansion of
	\begin{align*}
	\prod_{j\in S^C} A_{i_ji_{j+1}} = \prod_{j\in S^C} \Bigl(\sum_{\substack{1\leq k_j\leq M:\\ i_j,i_{j+1}\in \mathcal{I}(k_j)}}\tb g_{k_j,i_j}g_{k_j,i_{j+1}}\Bigr)
	\end{align*}
	is of the form \begin{align}\label{lem3:eqn1}
	(\tb)^{\ell-|S|} \prod_{j\in S^C}g_{k_j,i_j}g_{k_j,i_{j+1}}
	\end{align}
	for some $k_j = k_j(i_j,i_{j+1})$. In each term of this form, any index $i \neq 1, 2$ appears an even number of times as a `vertex index' (i.e., the second coordinate of the disorder $g$), while the indices 1 and 2 appear an odd number of times as the `vertex index' since we must have $i_{\min S^C} = i_0 = 1$ and $i_{\max S^C + 1} = i_{\ell} = 2$. Thus, every term of the form \eqref{lem3:eqn1} can be expressed as a function $Q(g)$ as described in the lemma.
\end{proof}

\begin{proof}[\bf Proof of Lemma \ref{lem0}]
	We only handle $\e A_{12}^{-1}A_{13}^{-1}=0$ as the arguments for the other two statements are the same. First of all, we consider the case that $\mathcal{D}$ is bounded. Since $A\geq I$, we can express
	$$
	A^{-1}=\int_0^\infty e^{-tA}dt.
	$$
	Also, since  $e^{-tA}\leq e^{-tI},$ we can bound
	$$
	|(e^{-tA})_{ij}|\leq \sqrt{(e^{-tA})_{ii}(e^{-tA})_{jj}}\leq \sqrt{(e^{-tI})_{ii}(e^{-tI})_{jj}}\leq e^{-t}.
	$$ 
	Consequently, from the Fubini theorem,
	\begin{align*}
		\e A_{12}^{-1}A_{13}^{-1}&=\int_0^\infty \int_0^\infty\e (e^{-tA})_{12}(e^{-sA})_{13}dtds.
	\end{align*}
	Our assertion would hold as long as  we establish that $\e (e^{-tA})_{12}(e^{-sA})_{13}=0$ for all $s,t\geq 0$. It suffices to show that $\e A_{12}^\ell A_{13}^{\ell'}=0$ for all $\ell,\ell'\geq 0.$ To see this, from Lemma \ref{lem3}, conditionally on $M$ and $\mathcal{I}$, we expand $A_{12}^{\ell}$ and $A_{13}^{\ell'}$ as sums of the terms of the forms $P(g)Q(g)$ and $P'(g)Q'(g)$ respectively, where $P$ and $P'$ are polynomials and even in each coordinate and 
	\begin{align*}
		Q(g)&=\prod_{k=1}^M\prod_{i=1}^Ng_{ki}^{d_k(i)}\,\,\mbox{and}\,\,Q'(g)=\prod_{k=1}^M\prod_{i=1}^Ng_{ki}^{d_k'(i)},
	\end{align*}
	where $\sum_{k=1}^Md_k(1),$ $\sum_{k=1}^Md_k(2),$ $\sum_{k=1}^Md_{k}'(1),$ and $\sum_{k=1}^Md_k'(3)$ are all odd and $\sum_{k=1}^Md_k(i)$ for $i\neq 1,2$ and $\sum_{k=1}^Md_k'(i)$ for $i\neq 1,3$ are all even. Consequently, from the symmetry of $\mathcal{D}$, conditionally on $M$ and $I(k,r)$'s, if we replace $(g_{k2})_{1\leq k\leq M}$ by $(-g_{k2})_{1\leq k\leq M}$, then 
	\begin{align*}
		P(g)Q(g)P'(g)Q'(g)&\stackrel{d}{=}P(g)P'(g)\Bigl(\prod_{k=1}^M(-g_{k2})^{d_k(2)+d_k'(2)}\Bigr)\Bigl(\prod_{i\neq 2}\prod_{k=1}^Mg_{k1}^{d_k(i)+d_k'(i)}\Bigr)\\
		&=-P(g)P'(g)\Bigl(\prod_{k=1}^Mg_{k2}^{d_k(2)+d_k'(2)}\Bigr)\Bigl(\prod_{i\neq 2}\prod_{k=1}^Mg_{k1}^{d_k(i)+d_k'(i)}\Bigr)\\
		&=-P(g)Q(g)P'(g)Q'(g),
	\end{align*}
	where we used the symmetry in $P,P'$ and the fact that $\sum_{k=1}^Md_k(2)$ is odd and $\sum_{k=1}^Md_k'(2)$ is even.
	Hence, by the boundedness of $\mathcal{D}$ and the dominated convergence theorem, $\e \bigl[P(g)Q(g)P'(g)Q'(g)\bigr]=0$ and thus, $\e A_{12}^{\ell}A_{13}^{\ell'}=0.$  
	
	Next the validity $\e A_{12}^{-1}A_{13}^{-1}=0$ without the boundedness of $\mathcal{D}$ can be obtained by first truncating $\mathcal{D}$ to be supported on $[-c,c]$, noting that the corresponding $A^c$ satisfies $|((A^c)^{-1})_{ij}|\leq 1$ for any $c$, and then sending $c$ to infinity in $\e ((A^c)^{-1})_{12}((A^c)^{-1})_{13}=0$ by the dominated convergence theorem. This completes our proof.
\end{proof}

\subsection{Concentration of linear statistics of eigenvalues}

Let $F$ be a symmetric matrix. For any real-valued continuous function $f$ on $\mathbb{R}$,  $f(F):=Uf(\Sigma)U^T$ is a well-defined symmetric matrix, where 
$F$ has the spectral  decomposition $U\Sigma U^T$ with  $U$ being an orthonormal matrix and $\Sigma$ being a diagonal matrix of eigenvalues of $F$. In all our cases, the matrix $F$ will be non-negative definite and hence we are allowed to define $f(F)$ whenever $f$ is defined on $\mathbb{R}_+$.
Denote $L=2\beta\sum_{k\leq M}v_kv_k^T$.
The following lemma controls the variance of the trace of $f(L)$.

\begin{lemma}\label{lem1}
	For any continuous function $f$ of bounded variation defined on $\mathbb{R}_+$, we have that 
	\begin{align*}
		\frac{1}{N^2}\mbox{\rm Var}(\mbox{\rm tr} f(L))\leq \frac{8\alpha\|f\|_{\mbox{\tiny BV}}^2}{N}.
	\end{align*}
\end{lemma}

\begin{proof}
	We recall the rank inequality for the empirical spectral distribution of matrices (see, e.g., \cite{bordenave2019lecture}): for $M_1, M_2$ two $N\times N$ Hermitian matrices with $\mbox{rank}(M_1 - M_2)=r,$ 
	\begin{align}\label{empiricalsd}
		\Bigl|\int f(x)\mu_{M_1}(dx)-\int f(x)\mu_{M_2}(dx)\Bigr|\leq \frac{r}{N}\|f\|_{\mbox{\tiny BV}},
	\end{align}
	for any continuous function $f$ of bounded variation, where $\mu_{M_i}$ is the empirical spectral distribution of the matrix $M_i$.  
	
	Now, let $f$ be any continuous function of bounded variation.
	For any $1\leq k\leq M$, let $v_k'$ be an independent copy of $v_k$ and let $$L_k=L-2\beta (v_kv_k^T-v_k'{v_k'}^T).$$ In other words, $L_k$ is obtained from $L$ by replacing  $v_kv_k^T$ by $v_k'{v_k'}^T$. Since $\mbox{rank}(L-L_k)\leq 2,$ the rank inequality implies that
	\begin{align*}
		\Bigl|\frac{1}{N}\mbox{tr}f(L)-\frac{1}{N}\mbox{tr}f(L_k)\Bigr|&=	\Bigl|\int f(x)\mu_{L}(dx)-\int f(x)\mu_{L_k}(dx)\Bigr|\leq \frac{2\|f\|_{\mbox{\tiny BV}}}{N}.
	\end{align*}
	Consequently, from the Efron-Stein inequality,
	\begin{align}
		\label{lem1:proof:eq1}
		\frac{1}{N^2}\mbox{Var}(\mbox{tr}f(L)|M)&\leq \frac{2M\|f\|_{\mbox{\tiny BV}}^2}{N^2}.
	\end{align}
	Similarly, conditionally on $g$ and $\mathcal{I},$ if $M'$ is an independent copy of $M$ and $L':=\sum_{k\leq M'}v_kv_k^T$, then $\mbox{rank}(L-L')\leq |M'-M|$ and
	\begin{align*}
		\Bigl|\frac{1}{N}\mbox{tr}f(L)-\frac{1}{N}\mbox{tr}f(L')\Bigr|=	\Bigl|\int f(x)\mu_{L}(dx)-\int f(x)\mu_{L'}(dx)\Bigr|\leq \frac{\|f\|_{\mbox{\tiny BV}}|M-M'|}{N}
	\end{align*}
	and from Jensen's inequality,
	\begin{align*}
		\frac{1}{N^2}	\mbox{Var}(\mbox{tr}f(L)|g,\mathcal{I})&\leq \frac{\|f\|_{\mbox{\tiny BV}}^2\e |M-M'|^2}{N^2}=\frac{2\alpha\|f\|_{\mbox{\tiny BV}}^2}{N}.
	\end{align*}
	This and \eqref{lem1:proof:eq1} together completes the proof since
	\begin{align*}
		\mbox{Var}(\mbox{tr}f(L))&\leq 2\e \Bigl[\mbox{Var}(\mbox{tr}f(L)|M)\Bigr]+2\e\Bigl[	\mbox{Var}(\mbox{tr}f(L)|g,\mathcal{I})\Bigr].
	\end{align*}
\end{proof}

\subsection{Proof of Theorem \ref{thm:concentration}}

In view of \eqref{add:eq7}, the proof of Theorem \ref{thm:concentration} follows directly from Lemmas \ref{lem2} and \ref{lem4} below.
\begin{lemma}\label{lem2}
	There exists a constant $K>0$ independent of $N$ such that
	\begin{align*}
		\e\Bigl|\frac{1}{N}\mathbbm{1}^TA^{-1}\mathbbm{1}-\frac{1}{N}\e\mathbbm{1}^TA^{-1}\mathbbm{1}\Bigr|^2&\leq \frac{K}{N}.
	\end{align*}
\end{lemma}

\begin{proof}
	From \eqref{lem0:eq1}, we can write
	\begin{align}
		\nonumber	\e\Bigl|\frac{1}{N}\mathbbm{1}^TA^{-1}\mathbbm{1}-\frac{1}{N}\e\mathbbm{1}^TA^{-1}\mathbbm{1}\Bigr|^2&=\e \Bigl|\frac{1}{N}\sum_{i=1}^N(A_{ii}^{-1}-\e A_{ii}^{-1})+\frac{1}{N}\sum_{i\neq j}A_{ij}^{-1}\Bigr|^2\\
		\label{lem0:proof:eq1}	&\leq2\e \Bigl|\frac{1}{N}\sum_{i=1}^N(A_{ii}^{-1}-\e A_{ii}^{-1})\Bigr|^2+2\e\Bigl|\frac{1}{N}\sum_{i\neq j}A_{ij}^{-1}\Bigr|^2.
	\end{align}
	To control the second term, note that $0\leq A_{11}^{-2}\leq 1$ so that
	\begin{align*}
		1\geq \e A_{11}^{-2}=\sum_{i=1}^N\e A_{1i}^{-1}A_{i1}^{-1}=\e |A_{11}^{-1}|^2+(N-1)\e |A_{12}^{-1}|^2,
	\end{align*}
	which implies that $(N-1)\e |A_{12}^{-1}|^2\leq 1$ and thus, this together with \eqref{lem0:eq2} and \eqref{lem0:eq3} yields that
	\begin{align*}
		\e\Big|\frac{1}{N}\sum_{i\neq j}A_{ij}^{-1}\Bigr|^2&=\frac{1}{N^2}\sum_{i\neq j}\e|A_{ij}^{-1}|^2=\frac{N-1}{N}\e |A_{12}^{-1}|^2\leq \frac{1}{N}.
	\end{align*}
	Next we handle the first term of \eqref{lem0:proof:eq1}. Note that if we take $f(x)=1/(1+x)$ for $x\in \mathbb{R}_+,$ then $f(L)=A^{-1}$ and from Lemma \ref{lem1} and noting that $\|f\|_{\mbox{\tiny BV}}=1$, 
	$$
	\e \Bigl|\frac{1}{N}\sum_{i=1}^N(A_{ii}^{-1}-\e A_{ii}^{-1})\Bigr|^2=\frac{1}{N^2}\mbox{Var}(\mbox{tr} A^{-1})\leq \frac{16\alpha}{N}.
	$$
	This completes our proof.
\end{proof}

\begin{lemma}\label{lem4}
	There exists a constant $K>0$ independent of $N$ such that
	\begin{align*}
		\e\Bigl|\frac{1}{N}\log \det A-\frac{1}{N}\e\log \det A\Bigr|\leq \frac{K}{\sqrt{N}}.
	\end{align*}
\end{lemma}

\begin{proof}
	For any $c>0,$ let $f_c(x)=\log (1+x \wedge c)$ if $x \ge 0$. Recall $L$ from Lemma \ref{lem1} and let $\lambda_i(L)$ be the eigenvalues of $L.$ Write
	\begin{align*}
		\log   \det (I+L)
		&=\mbox{tr}f_c(L)+\sum_{i:\lambda_i(L)>c}\log(1+\lambda_i(L))-\log (1+c)\#\{i:\lambda_i(L)>c\}
	\end{align*}
	so that 
	\begin{align*}
		& \Bigl|	\frac{1}{N}\log   \det (I+L)-\frac{1}{N}\mbox{tr}f_c(L)\Bigr|\\		
		&\leq \frac{1}{N}\sum_{i:\lambda_i(L)> c}\log  (1+\lambda_i(L))+\frac{\log(1+c)}{N}\#\{i:\lambda_i(L)>c\}\\
		&\leq  \frac{1}{N}\sum_{i:\lambda_i(L)> c}\lambda_i(L)^{1/2}+\frac{\log2(1+c)}{N}\#\{i:\lambda_i(L)>c\},
	\end{align*}
	where the second inequality used $1+x\leq 2e^{x^{1/2}}$ for $x>0.$	To control this inequality, from the Cauchy-Schwarz and Markov inequalities, we obtain,
	\begin{align*}
		\frac{1}{N}\e\#\{i:\lambda_i(L)>c\}\leq \frac{\e\mbox{tr} L}{cN}
	\end{align*}
	and
	\begin{align*}
		\frac{1}{N}\e\sum_{i:\lambda_i(L)> c}\lambda_i(L)^{1/2}&\leq \e\Bigl(\frac{\#\{i:\lambda_i(L)>c\}}{N}\Bigr)^{1/2}\Bigl(\frac{\mbox{tr}L}{N}\Bigr)^{1/2}\leq \frac{\e\mbox{tr}L}{\sqrt{c}N}.
	\end{align*}
	Since
	\begin{align*} 
		\frac{1}{N}\e\mbox{tr}L&=\frac{2\beta}{N}\e\sum_{k=1}^M\sum_{r= 1}^p g_{k,I(k,r)}^2=2\beta p\alpha \e g_{1,1}^2, 
	\end{align*}
	putting these inequalities together yields that
	\begin{align}\label{eq:ctr1}
		\e \Bigl|	\frac{1}{N}\log   \det (I+L)-\frac{1}{N}\mbox{tr}f_c(L)\Bigr|&\leq 2\beta\alpha p\e g_{1,1}^2\Bigl(\frac{1}{\sqrt{c}}+\frac{\log 2(1+c)}{c}\Bigr)
	\end{align} 
	On the other hand, note that $\|f_c\|_{BV}=\log  (1+c)$ so that from Lemma \ref{lem1},
	\begin{align*}
		%\label{eq:ctr2}
		\e \Bigl|\frac{1}{N}\mbox{tr}f_c(L)-\frac{1}{N}\e\mbox{tr}f_c(L)\Bigr|\leq \frac{8\alpha\log^2(1+c)}{N}.
	\end{align*}
	Plugging $c=N$ into these two inequalities yields the desired result.
\end{proof}

\section{Concentration of the generalized multi-overlap}\label{sec:conop}

Let $\la\cdot\ra^c$ be the Gibbs measure associated to the partition function
$
\int e^{-X_N(\tau)}\eta_N(d\tau),
$
where $$
X_N(\tau):=\beta \sum_{k=1}^M\Bigl(\sum_{r=1}^p g_{k,I(k,r)}\tau_{I(k,r)}\Bigr)^2.
$$
The $X_N$ here is different from $H_N$ by dropping out the external field term, i.e., $h=0.$
While $\sigma\thicksim \la \cdot\ra$ is $N(\mu,A^{-1})$ as explained in Section \ref{sec:concentration}, we see that $\tau\thicksim \la \cdot\ra^c$ is $N(0,A^{-1}).$

Let $\kappa\geq 1$  and $f:\mathbb{R}^\kappa\to[-1,1]$ be a smooth function satisfying that
\begin{align}\label{sec:conop:eq1}
	|f(x)-f(x')|\leq \|x-x'\|,\,\,\forall x,x'\in \mathbb{R}^\kappa
\end{align}
and
\begin{align}\label{sec:conop:eq2}
	-\frac{1}{4}I\leq \Delta^2f(x)\leq \frac{1}{4}I,\,\,\forall x\in \mathbb{R}^\kappa.
\end{align} 
For i.i.d. samples $\tau^1,\ldots,\tau^\kappa$ from $\la \cdot\ra^c,$ denote $\vec\tau =(\tau^1,\ldots,\tau^\kappa)$ and $$\|\vec\tau\|=\sqrt{\|\tau^1\|^2+\cdots+\|\tau^\kappa\|^2}.$$ Define the generalized multi-overlap by
\begin{align*}
	Q=Q(\vec\tau)=\frac{1}{N}\sum_{i=1}^Nf(\tau_i^1,\ldots,\tau_i^\kappa).
\end{align*}
The main result of this section is the following concentration. 

\begin{prop}\label{sec:conop:prop1}
	Assume that $\mathcal{D}$ is bounded. There exists a constant $K>0$ (depending on $\alpha, \beta, \kappa$ and $\mathcal{D}$) such that 
	\begin{align*}
		\e \bigl\la \bigr|Q-\e\bigl\la Q\bigr\ra^c|\bigr\ra^c\leq \frac{K}{N^{1/4}},\,\,\forall N\geq 1.
	\end{align*}	
\end{prop}
For the rest of this section, we establish this proposition in three subsections. Throughout our entire argument as well as in the statements of the lemmas, the letters $K,K',K'',$ etc. stand for deterministic constants dependent on $\alpha, \beta, \kappa$ and $\mathcal{D}$ and independent of $N$, but they might be different from each occasion.

\subsection{Moment controls}\label{sec:momentcontrol}

For any $|\lambda|\leq 1$, define an auxiliary free energy by
\begin{align*}
	F_N^c(\lambda)&=\frac{1}{N}\log \int e^{-\sum_{\ell=1}^\kappa X_N(\tau^\ell)+\lambda NQ(\vec\tau)}\eta_N(d\vec\tau)
\end{align*}
Denote by $\la \cdot\ra_{\lambda}^c$ the Gibbs expectation associated to $F_N^c(\lambda)$. In this section, we establish two moment bounds for the spin configuration sampled from $\la\cdot\ra_\lambda^c.$

\begin{lemma}\label{add:lem3}
	Assume that $\mathcal{D}$ is bounded.	For any $n\geq 1,$ there exists a constant $K>0$  such that for all $|\lambda|\leq 1$ and $N\geq 1,$
	\begin{align*}
		\la \|\vec\tau\|^{2n}\ra_{\lambda}^c&\leq KN^n\Bigl(1+\frac{M}{N}\Bigr)^n.
	\end{align*}
\end{lemma}

\begin{proof}
	First of all, we claim that	 there exists a constant $K$ such that 
	\begin{align*}
		\la e^{\frac{1}{4}\|\vec\tau\|^2}\ra_{\lambda}^c\leq e^{K(\sum_{k=1}^M\|v_k\|^2+N)},
	\end{align*}
	where we recall that $v_k$'s are the vectors defined in \eqref{def_A}. To see this, note that
	\begin{align*}
		\int e^{\frac{1}{4}\|\vec\tau\|^2-\sum_{\ell=1}^\kappa X_N(\tau^\ell)+\lambda NQ(\vec\tau)}\eta_N(d\vec\tau)&\leq \frac{1}{(2\pi)^{\kappa N/2}}\int e^{\frac{1}{4}\|\vec\tau\|^2-\frac{1}{2}\|\vec\tau\|^2+|\lambda|N}d\vec\tau\\
		&=\frac{e^{|\lambda|N}}{(2\pi)^{\kappa N/2}}\int e^{-\frac{1}{4}\|\vec\tau\|^2}d\vec\tau\leq K^N
	\end{align*} for some $K>0$ independent of $\lambda.$ On the other hand, by Jensen's inequality,
	\begin{align*}
		\log\int e^{-\sum_{\ell=1}^\kappa X_N(\tau^\ell)+\lambda NQ(\vec\tau)}\eta_N (d\vec\tau)&\geq -\int \sum_{\ell=1}^\kappa X_N(\tau^\ell)\eta_N(d\vec\tau)-N\\
		&=-\beta \kappa \sum_{k=1}^M\|v_k\|^2-N.
	\end{align*}
	Putting these two inequalities together yields our claim. Finally, recall from \cite[Lemma 3.1.8]{Tal111} that  for any nonnegative random variable $Y,$ 
	$$
	\E Y^n\leq 2^n\bigl(n^n+\bigl(\log \E\exp Y\bigr)^n\bigr),\,\,n=1,2,3,\ldots.
	$$
	Our proof is completed by applying this inequality with $Y= \|\vec\tau\|^{2}$ and using the boundedness of $\mathcal{D}$.
\end{proof}

\begin{lemma}\label{add:lem1}
	Assume that $\mathcal{D}$ is bounded.  There exists a constant $K>0$ such that for any $N\geq 1$ and $|\lambda|\leq 1/N^{1/4}$, 
	\begin{align*}
		\e\bigl[\bigl\la |\tau_i^\ell|^4\bigr\ra_\lambda^c\big|M\bigr]\leq K\Bigl(1+\Bigl(\frac{M}{N}\Bigr)^2\Bigr)
	\end{align*}
	for all $1\leq i\leq N$ and $1\leq \ell\leq \kappa$.
\end{lemma}

\begin{proof}
	For any fixed $1\leq \ell\leq \kappa,$ by symmetry among the spins $\tau_1^\ell,\ldots,\tau_N^\ell$, it can be seen that $\e\bigl[\bigl\la |\tau_i^\ell|^4\bigr\ra_\lambda^c\big|M\bigr]$ are all the same for $1\leq i\leq N.$ Although they might vary in $\ell,$ our proof does not distinguish between different $\ell$'s and we shall establish our assertion only for $\e\bigl[\bigl\la |\tau_1^1|^4\bigr\ra_\lambda^c\big|M\bigr]$. To begin with, note that for any smooth $\phi$ of moderate growth, i.e., $\lim_{\|\vec x\|\to\infty}|\phi(\vec x)|e^{-t\|\vec x\|^2}=0$ for all $t>0$, we can write
	\begin{align*}
		\la \tau_1^1\phi(\vec\tau)\ra_\lambda^c&=\frac{\bigl\la\tau_1^1\phi(\vec\tau)e^{\lambda NQ(\vec\tau)}\bigr\ra^c}{\bigl\la e^{\lambda NQ(\vec\tau)}\bigr\ra^c}.
	\end{align*}
	In the numerator, since $\tau^1,\ldots,\tau^\kappa$ are i.i.d. samples from $N(0,A^{-1})$, we can apply the Gaussian integration by parts for $\tau_1^1$ to see that
	\begin{align*}
		\bigl\la\tau_1^1\phi(\vec\tau)e^{\lambda NQ(\vec\tau)}\bigr\ra^c
		&=	\sum_{j=1}^NA_{1j}^{-1}\Bigl\la \Bigl(\partial_{\tau_j^1}\phi(\vec\tau)+\lambda N\phi(\vec\tau)\partial_{\tau_j^1}Q(\vec\tau)\Bigr)e^{\lambda NQ(\vec\tau)}\Bigr\ra^c\\
		&=	\sum_{j=1}^NA_{1j}^{-1}\Bigl\la \Bigl(\partial_{\tau_j^1}\phi(\vec\tau)+\lambda \phi(\vec\tau)\partial_{x_1}f(\vec\tau_j)\Bigr)e^{\lambda NQ(\vec\tau)}\Bigr\ra^c
	\end{align*} 
	with the understanding that $\vec\tau_j=(\tau_j^1,\ldots,\tau_j^\kappa).$ Consequently,
	\begin{align}\label{add:eq-10}
		\la \tau_1^1\phi(\vec\tau)\ra_\lambda^c&=\sum_{j=1}^NA_{1j}^{-1}\Bigl\la \partial_{\tau_j^1}\phi(\vec\tau)\Bigr\ra_\lambda^c+\lambda\Bigl\la \phi(\vec\tau)\sum_{j=1}^NA_{1j}^{-1}\partial_{x_1}f(\vec\tau_j)\Bigr\ra_\lambda^c.
	\end{align}
	In particular, we have
	\begin{align*}
		\la |\tau_1^1|^4\ra_\lambda^c&=3A_{11}^{-1}\la |\tau_1^1|^2\ra_\lambda^c+\lambda\Big\la (\tau_1^1)^3 \sum_{j=1}^NA_{1j}^{-1}\partial_{x_1}f(\vec\tau_j)\Bigr\ra_{\lambda}^c.
	\end{align*} 
	Applying \eqref{add:eq-10} one more time for the second term gives
	\begin{align*}
		\la |\tau_1^1|^4\ra_\lambda^c	&=3A_{11}^{-1}\la |\tau_1^1|^2\ra_\lambda^c+2\lambda A_{11}^{-1}\Big\la \tau_1^1 \sum_{j=1}^NA_{1j}^{-1}\partial_{x_1}f(\vec\tau_j)\Bigr\ra_{\lambda}^c\\
		&\qquad+\lambda\Big\la |\tau_1^1|^2\sum_{j=1}^N|A_{1j}^{-1}|^2\partial_{x_1x_1}f(\vec\tau_j)\Bigr\ra_{\lambda}^c+\lambda^2\Big\la |\tau_1^1|^2\Bigl|\sum_{j=1}^NA_{1j}^{-1}\partial_{x_1}f(\vec\tau_j)\Bigr|^2\Bigr\ra_{\lambda}^c.
	\end{align*}
	Observe that if we switch $\tau_1^1$ to $\tau_i^1$, the same formula remains valid with the only change that $\tau_1^1A_{11}^{-1}$ is replaced by $\tau_i^1A_{ii}^{-1}$ and $A_{1j}^{-1}$ is replaced by $A_{ij}^{-1}.$ Consequently, adding $\la |\tau_1^1|^4\ra_\lambda^c,\ldots,\la |\tau_N^1|^4\ra_\lambda^c$ together and using symmetry among the spins, we arrive at
	\begin{align}
		\label{add:lem1:proof:eq2}N\e\la |\tau_1^1|^4\ra_\lambda^c&=3\e\sum_{i=1}^NA_{ii}^{-1}\la |\tau_i^1|^2\ra_\lambda^c+2\lambda\e\Big\la \sum_{i,j=1}^N A_{ij}^{-1}(A_{ii}^{-1}\tau_i^1)\partial_{x_1}f(\vec\tau_j)\Bigr\ra_{\lambda}^c\\
		\label{add:lem1:proof:eq3}&\qquad +\lambda\e\Big\la \sum_{i,j=1}^N|A_{ij}^{-1}|^2|\tau_i^1|^2\partial_{x_1x_1}f(\vec\tau_j)\Bigr\ra_{\lambda}^c\\
		\label{add:lem1:proof:eq4}& \qquad +\lambda^2\e\Big\la \sum_{i=1}^N\Bigl|\tau_i^1\sum_{j=1}^NA_{ij}^{-1}\partial_{x_1}f(\vec\tau_j)\Bigr|^2\Bigr\ra_{\lambda}^c.
	\end{align}
	We control these terms as follows. To simplify our notation, we denote $\e[\,\cdot\,|M]$ by $\e[\,\cdot\,].$ 
	Recall that $0\leq A_{ii}^{-1}\leq 1$ and $\|A^{-1}\|\leq 1.$ It follows that \eqref{add:lem1:proof:eq2} can be controlled by using Lemma \ref{add:lem3}, 
	\begin{align}
		\label{add:lem1:proof:eq1}\e\sum_{i=1}^NA_{ii}^{-1}\la |\tau_i^1|^2\ra_\lambda^c&\leq \e \bigl\la \|\tau^1\|^2\bigr\ra_\lambda^c\leq K(M+N),\\
		\label{add:lem1:proof:eq1.5}\e\Big\la \Bigl|\sum_{i,j=1}^N A_{ij}^{-1}(A_{ii}^{-1}\tau_i^1)\partial_{x_1}f(\vec\tau_j)\Bigr|\Bigr\ra_{\lambda}^c&\leq \e \bigl\la \|\tau^1\|\|\partial_{x_1}f(\vec\tau)\|\bigr\ra_\lambda^c\leq K\sqrt{N(M+N)},
	\end{align}
	where $\partial_{x_1}f(\vec\tau): = (\partial_{x_1}f(\vec\tau_1),\ldots, \partial_{x_1}f(\vec\tau_N))$. Next, since the Hadamard  product $A^{-1}\circ A^{-1}$ satisfies $\|A^{-1}\circ A^{-1}\|\leq \|A^{-1}\|^2\leq 1,$ \eqref{add:lem1:proof:eq3} can be handled by
	\begin{align}
		\nonumber&\e\Big\la \sum_{i,j=1}^N|A_{ij}^{-1}|^2|\tau_i^1|^2\partial_{x_1x_1}f(\vec\tau_j)\Bigr\ra_{\lambda}^c\\
		\nonumber&\leq \e\|A^{-1}\circ A^{-1}\|\Bigl\la\Bigl(\sum_{i=1}^N|\tau_i^1|^4\Bigr)^{1/2}\Bigl(\sum_{j=1}^N|\partial_{x_1x_1}f(\vec\tau_j)|^2\Bigr)^{1/2}\Bigr\ra_{\lambda}^c\\
		\nonumber&\leq K\sqrt{N}\Bigl(\e\Bigl\la\sum_{i=1}^N|\tau_i^1|^4\Bigr\ra_{\lambda}^c\Bigr)^{1/2}\\
		\label{add:lem1:proof:eq5}&=KN\bigl(\e \bigl\la |\tau_1^1|^4\bigr\ra_\lambda^c\bigr)^{1/2}.
	\end{align}
	To handle \eqref{add:lem1:proof:eq4}, we use the identity $\|x\|^2=\sup_{\|a\|=1}(a^Tx)^2$ to write
	\begin{align*}
		\e\Big\la \sum_{i=1}^N\Bigl|\tau_i^1\sum_{j=1}^NA_{ij}^{-1}\partial_{x_1}f(\vec\tau_j)\Bigr|^2\Bigr\ra_{\lambda}^c&=\e\Big\la \sup_{\|a\|=1}\Bigl(\sum_{i,j=1}^NA_{ij}^{-1}(a_i\tau_i^1)\partial_{x_1}f(\vec\tau_j)\Bigr)^2\Bigr\ra_{\lambda}^c\\
		&\leq KN\e\Big\la \sup_{\|a\|=1}\Bigl|\sum_{i=1}^N(a_i\tau_i^1)^2\Big|\Bigr\ra_{\lambda}^c,
	\end{align*}
	Here, the last inequality can be controlled by using the Cauchy-Schwarz inequality,
	\begin{align*}
		\e\Bigl\la	\sup_{\|a\|=1}\Bigl|\sum_{i=1}^N(a_i\tau_i^1)^2\Big|\Bigr\ra_{\lambda}^c&\leq \e\Bigl\la\Bigl(\sum_{i=1}^N|\tau_i^1|^4\Bigr)^{1/2}\Big\ra_{\lambda}^c\\
		&\leq \Bigl(\e\Bigl\la\sum_{i=1}^N|\tau_i^1|^4\Big\ra_{\lambda}^c\Bigr)^{1/2}=\sqrt{N}\bigl(\e\bigl\la|\tau_1^1|^4\bigr\ra_{\lambda}^c\bigr)^{1/2},
	\end{align*}
	which implies that
	\begin{align}\label{add:lem1:proof:eq6}
		\e\Big\la \sum_{i=1}^N\Bigl|\tau_i^1\sum_{j=1}^NA_{ij}^{-1}\partial_{x_1}f(\vec\tau_j)\Bigr|^2\Bigr\ra_{\lambda}^c&\leq KN\sqrt{N}\bigl(\e\bigl\la|\tau_1^1|^4\bigr\ra_{\lambda}^c\bigr)^{1/2}.
	\end{align}
	Now combining \eqref{add:lem1:proof:eq1}, \eqref{add:lem1:proof:eq1.5}, \eqref{add:lem1:proof:eq5}, and \eqref{add:lem1:proof:eq6} together, for any $|\lambda|\leq 1/N^{1/4},$
	\begin{align*}
		\e\bigl\la|\tau_1^1|^4\bigr\ra_{\lambda}^c&\leq K(1+M/N)+K(\lambda\sqrt{1+M/N}+\lambda^2\sqrt{N})\bigl(\e\bigl\la|\tau_1^1|^4\bigr\ra_{\lambda}^c\bigr)^{1/2}\\
		&\leq K(1+M/N)+K'\sqrt{1+M/N}\bigl(\e\bigl\la|\tau_1^1|^4\bigr\ra_{\lambda}^c\bigr)^{1/2}
	\end{align*} 
	for some constant $K'$ independent of $N.$ This inequality readily implies that 
	\begin{align*}
		\e\bigl\la|\tau_1^1|^4\bigr\ra_{\lambda}^c&\leq K''\bigl(1+(M/N)^2\bigr).
	\end{align*} 
\end{proof}

\subsection{Concentration of the free energy}

We proceed to show that the auxiliary free energy $F_N^c(\lambda)$ is concentrated by using Subsection \ref{sec:momentcontrol}.

\begin{prop}\label{add:prop1}
	Assume that $\mathcal{D}$ is bounded. There exists a constant $K>0$ such that for any $N\geq 1$ and $|\lambda|\leq 1/N^{1/4}$, 
	\begin{align*}
		\e \bigl|F_N^c(\lambda)-\e F_N^c(\lambda)\bigr|&\leq \frac{K}{\sqrt{N}}.
	\end{align*}
\end{prop}

The rest of this subsection establishes this proposition. Recall $\mathcal{I}$ from \eqref{add:eq22}. For any random variable or random vector $X,$ $\e_X$ stands for the expectation in $X$ only. 
Now, write by using the Jensen inequality,
\begin{align*}
	\e \bigl|F_N^c(\lambda)-\e F_N^c(\lambda)\bigr|&\leq \e \bigl|F_N^c(\lambda)-\e_M F_N^c(\lambda)\bigr|+\e\bigl|\e_MF_N^c(\lambda)-\e_M \e_{g,\mathcal{I}}F_N^c(\lambda)\bigr|\\
	&\leq \e \bigl|F_N^c(\lambda)-\e_M F_N^c(\lambda)\bigr|+\e \bigl|F_N^c(\lambda)-\e_{g,\mathcal{I}}F_N^c(\lambda)\bigr|.
\end{align*}
The proof of Proposition \ref{add:prop1} is completed by the following two lemmas. 

\begin{lemma}
	Assume that $\mathcal{D}$ is bounded.	There exists a constant $K>0$ such that for any $N\geq 1$ and $|\lambda|\leq 1/N^{1/4}$, 
	\begin{align*}
		\e\bigl|F_N^c(\lambda)-\e_M  F_N^c(\lambda)\bigr|\leq \frac{K}{\sqrt{N}}.
	\end{align*}
\end{lemma}

\begin{proof}Let $\hat M$ be an independent copy of $M$. Let $\hat X_N$ and $\hat F_N^c(\lambda)$ be equal to $X_N$ and $F_N^c(\lambda)$ with the replacement of $M$ by $\hat M.$ Assume that $M\leq \hat M.$ Since
	$X_N\leq \hat X_N$, we have $\hat F_N^c\leq F_N^c.$ On the other hand, 
	\begin{align*}
		\hat F_N^c(\lambda)-F_N^c(\lambda)&=\frac{1}{N}\log \Bigl\la\exp\Bigl(-\beta\sum_{\ell=1}^\kappa\sum_{M< k\leq \hat M}\Bigl(\sum_{r=1}^p g_{I(k,r)}\tau_{I(k,r)}^\ell\Bigr)^2\Bigr)\Bigr\ra_{\lambda}^c\\
		&\geq -\frac{\beta}{N}\sum_{\ell=1}^\kappa\sum_{M<k\leq M'}\Bigl\la\Bigl(\sum_{r=1}^pg_{k,I(k,r)}\tau_{I(k,r)}^\ell\Bigr)^2\Bigr\ra_\lambda^c\\
		%	&\geq -\frac{\beta p}{N}\sum_{\ell=1}^\kappa\sum_{M<k\leq \hat M}\sum_{r=1}^pg_{k,I(k,r)}^2\bigl\la |\tau_{I(k,r)}^\ell|^2\big\ra_{\lambda}^c\\
		&\geq -\frac{K}{N}\sum_{\ell=1}^\kappa\sum_{M<k\leq \hat M}\sum_{r=1}^p\bigl\la |\tau_{I(k,r)}^\ell|^2\big\ra_{\lambda}^c,
	\end{align*}
	where the first inequality used the Jensen inequality and the second inequality used the boundedness of $\mathcal{D}.$ 
	Since  conditionally on $M$ and $\hat M$ with $M<\hat M$, $\{I(k,r)\}_{M<k\leq \hat M,1\leq r\leq p}$ is independent of  $\{I(k,r)\}_{1\leq k\leq M,1\leq r\leq p}$ that appear in $\la \cdot\ra_\lambda^c$,  taking expectation and using symmetry in the spins implies that on the event $M<\hat M,$
	\begin{align*}
		\e\bigl[\bigl|\hat F_N^c(\lambda)-F_N^c(\lambda)\bigr|\big|M,\hat M\bigr]&\leq \frac{K}{N}(\hat M-M)p\sum_{\ell=1}^\kappa\e\bigl[\bigl\la |\tau_{1}^\ell|^2\big\ra_{\lambda}^c\bigl|M,\hat M\bigr]\\
		&=\frac{K}{N^2}(\hat M-M)p\e\bigl[\bigl\la \|\vec\tau\|^2\big\ra_{\lambda}^c\bigl|M,\hat M\bigr].
	\end{align*}
	Since this inequality is also valid if $M\geq \hat M$ with the obvious replacement of $\hat M-M$ by $M-\hat M$, we conclude that after using Lemma \ref{add:lem3},
	$$
	\e\bigl[\bigl|\hat F_N^c(\lambda)-F_N^c(\lambda)\bigr|\big|M,\hat M\bigr]\leq \frac{K'}{N}|\hat M-M|\Bigl(1+\frac{M}{N}\Bigr)\leq \frac{K'}{N}|\hat M-M|\Bigl(1+\frac{M+\hat M}{N}\Bigr).
	$$
	It follows from the Jensen inequality and the Cauchy-Schwarz inequality that
	\begin{align*}
		\e\bigl|\e_MF_N^c(\lambda)-F_N^c(\lambda)\bigr|
		&\leq \e\bigl|\hat F_N^c(\lambda)-F_N^c(\lambda)\bigr| \\
		&\leq \frac{K'}{N}\bigl(\e |\hat M-M|^2\bigr)^{1/2}\Bigl(1+\frac{\bigl(\e M^2\bigr)^{1/2}+\bigl(\e \hat M^2\bigr)^{1/2}}{N}\Bigr).
	\end{align*}
	Since $\e M^2=(\alpha N)^2+\alpha N$ and $\e|\hat M-M|^2=2\alpha N,$ 
	plugging these into the last display completes our proof.
\end{proof}

The next lemma controls $\e \bigl|F_N^c(\lambda)-\e_{g,\mathcal{I}}F_N^c(\lambda)\bigr|^2.$

\begin{lemma}\label{add:lem0}Assume that $\mathcal{D}$ is bounded.
	There exists a constant $K>0$ such that for any $N\geq 1$ and $|\lambda|\leq 1/N^{1/4}$,
	\begin{align*}
		\e\bigl|F_N^c(\lambda)-\e_{g,\mathcal{I}} F_N^c(\lambda)\bigr|^2\leq \frac{K}{N}.
	\end{align*}
\end{lemma}

\begin{proof}
	Let $(\tilde g_{1i})_{1\leq i\leq N}$ and $(\tilde I(1,r))_{1\leq r\leq p}$ be independent copies of $(g_{1i})_{1\leq i\leq N}$ and $(I(1,r))_{1\leq r\leq p}$, respectively. These are also independent of each other and everything else. Let $\tilde F_N^c(\lambda)$ be defined as $F_N^c(\lambda)$ except that we replace the component $$\sum_{\ell=1}^\kappa\Bigl(\sum_{r=1}^p g_{1,I(1,r)}\tau_{I(1,r)}^\ell\Bigr)^2$$
	in the Hamiltonian associated to $F_N^c(\lambda)$ by $$\sum_{\ell=1}^\kappa\Bigl(\sum_{r=1}^p\tilde g_{1,\tilde I(1,r)}\tau_{\tilde I(1,r)}^\ell\Bigr)^2.$$
	We claim that for any $M\geq 1,$
	\begin{align}\label{add:eq-12}
		\e \bigl[\bigl|F_N^c(\lambda)-\tilde F_N^c(\lambda)\bigr|^2\big|M\bigr]\leq  \frac{K}{N^2}\Bigl(1+\frac{M^2}{N^2}\Bigr).
	\end{align}
	To this end, let $\bar F_N^c(\lambda)$ be defined as $F_N^c(\lambda)$ except that we delete the component
	\begin{align*}
		%\label{add:eq-11}
		\sum_{\ell=1}^\kappa\Bigr(\sum_{r=1}^p g_{1,I(1,r)}\tau_{I(1,r)}^\ell \Bigr)^2
	\end{align*} from the Hamiltonian associated to $F_N^c(\lambda)$. Consequently, from the Jensen inequality,
	\begin{align*}
		0&\geq F_N^c(\lambda)-\bar F_N^c(\lambda)\geq -\frac{\beta}{N} \sum_{\ell=1}^\kappa\Bigl\la\Bigl(\sum_{r=1}^p g_{1,I(1,r)}\tau_{I(1,r)}^\ell\Bigr)^2\Bigr\ra_{\lambda}^-,\\
		0&\geq \tilde F_N^c(\lambda)-\bar F_N^c(\lambda)\geq -\frac{\beta}{N} \sum_{\ell=1}^\kappa\Bigl\la\Bigl(\sum_{r=1}^p\tilde g_{1,\tilde I(1,r)}\tau_{\tilde I(1,r)}^\ell\Bigr)^2\Bigr\ra_{\lambda}^-,
	\end{align*}
	where $\la \cdot\ra_\lambda^-$ is the Gibbs expectation corresponding to $\bar F_N^c(\lambda).$ Consequently, from the boundedness of $\mathcal{D}$ and Jensen's inequality,
	\begin{align*}
		\bigl|F_N^c(\lambda)-\tilde F_N^c(\lambda)\bigr|^2&\leq \frac{K}{N^2}\sum_{\ell=1}^\kappa\sum_{r=1}^p \bigl(\bigl\la |\tau_{I(1,r)}^\ell|^4\bigr\ra_\lambda^-+\bigl\la |\tau_{\tilde I(1,r)}^\ell|^4\bigr\ra_\lambda^-\bigr).
	\end{align*}
	From this, noting that $(I(1,j),\tilde I(1,j))_{1\leq j\leq p}$ do not appear in $\la \cdot\ra_\lambda^-$ and using the symmetry among the spins in each $\tau^\ell$ yield that whenever $M\geq 1,$
	\begin{align*}
		\e\bigl[\bigl|F_N^c(\lambda)-\tilde F_N^c(\lambda)\bigr|^2\big|M\bigr]&\leq \frac{K}{N^2}\sum_{\ell=1}^\kappa\sum_{r=1}^p \bigl(\e\bigl[\bigl\la |\tau_{1}^\ell|^4\bigr\ra_\lambda^-\big|M\bigr]+\e\bigl[\bigl\la |\tau_{1}^\ell|^4\bigr\ra_\lambda^-\big|M\bigr]\bigr).\\
		&\leq \frac{K'}{N^2}\Bigl(1+\Bigl(\frac{M-1}{N}\Bigr)^2\Bigr)\leq \frac{K'}{N^2}\Bigl(1+\frac{M^2}{N^2}\Bigr),
	\end{align*}
	where the last inequality used Lemma \ref{add:lem1}. This completes the proof of our claim.
	
	We now turn to the proof of our assertion. Consider the filtration $(\mathcal{F}_s)_{s\geq1 }$ defined as $\mathcal{F}_0=\{\emptyset,\Omega\}$ and for $s\geq 1$, $\mathcal{F}_s=\sigma(g_{k, I(k, r)}:1\leq k\leq s,1\leq r\leq p)$. Define
	\begin{align*}
		d_s=\e \bigl[F_N^c(\lambda)|\mathcal{F}_s,M\bigr]-\e \bigl[F_N^c(\lambda)|\mathcal{F}_{s-1},M\bigr].
	\end{align*} 
	Then $\sum_{k=1}^{M}d_k=F_N^c(\lambda)-\e_{g,\mathcal{I}}F_N^c(\lambda)$. Since $(d_s)_{1\leq s\leq M}$ is a martingale difference, it follows that
	\begin{align}\label{lem4:proof:eq1}
		\e\bigl[ \bigl|F_N^c(\lambda)-\e_{g,\mathcal{I}}F_N^c(\lambda)\bigr|^2\big|M\bigr]&=\sum_{k=1}^{M}\e\bigl[ d_k^2\big|M\bigr].
	\end{align}
	Finally, from Jensen's inequality and symmetry, we see that for any $1\leq k\leq M,$
	\begin{align*}
		%\label{add:lem2:proof:eq1}
		\e\bigl[ d_k^2\big|M\bigr]&\leq \e \bigl[\bigl|F_N^c(\lambda)-\tilde F_N^c(\lambda)\bigr|^2\big|M\bigr].
	\end{align*}
	Consequently, from this inequality, \eqref{lem4:proof:eq1}, and \eqref{add:eq-12},
	\begin{align*}
		\e\bigl[ \bigl|F_N^c(\lambda)-\e_{g,\mathcal{I}}F_N^c(\lambda)\bigr|^2\bigr]&\leq \frac{K'}{N^2}\e\Bigl[M\Bigl(1+\frac{M^2}{N^2}\Bigr);M\geq 1\Bigr]\leq \frac{K''}{N}.
	\end{align*}
\end{proof}

\begin{remark}
	\rm The proof of Lemma \ref{add:lem0} heavily relies on the fourth moment bound in Lemma \ref{add:lem1}. Although it is not needed, whether Lemma \ref{add:lem1} also holds for fixed $\lambda$ or can be extended to higher moments remains elusive. 
\end{remark}

\begin{remark}
	\rm In view of the proof of Proposition \ref{sec:conop:prop1}, it might seem like that one can also prove the concentration for $F_N$ in Theorem \ref{thm:concentration} by the same approach for Proposition \ref{add:prop1}. Although ideally this should be the case, we point out that the missing ingredient in doing so is  an upper bound for $\e \la |\sigma_1|^4\ra$ similar to the one for $\e \la |\tau_i^\ell|^4\ra_\lambda^c$ in Lemma \ref{add:lem1}. To explain the main obstacle, we recall that as $\sigma_1\sim N(\mu_1,A_{11}^{-1})$, we have $\la |\sigma_1|^4\ra=\mu_1^4+6\mu_1^2(A_{11}^{-1})^2+3(A_{11}^{-1})^4$. Since $\mu_1=h\sum_{i=1}^NA_{1i}^{-1}$, it  seems to be a very challenging task to show that $\e\mu_1^4$ is of order $O(1).$ 
\end{remark}

\subsection{Proof of Proposition \ref{sec:conop:prop1}}

We need the following lemma:

\begin{lemma}[Thermal concentration]
	\label{add:lem4}
	There exists a deterministic constant $K>0$ such that for any $|\lambda|\leq 1$ and $N\geq 1$,
	\begin{align*}
		\bigl\la\bigr|Q-\bigl\la Q\bigr\ra_\lambda^c\bigr|^2\bigr\ra_\lambda^c\leq \frac{K}{N}.
	\end{align*}
\end{lemma}

\begin{proof}
	From \eqref{sec:conop:eq1}, 
	\begin{align*}
		\bigl|Q(\tau^1,\ldots,\tau^\kappa)-Q(\hat\tau^1,\ldots,\hat \tau^\kappa)\bigr|&\leq \frac{1}{N}\sum_{i=1}^N\Bigl(\sum_{\ell=1}^\kappa\bigl|\tau_i^\ell-\hat\tau_i^\ell|^2\Bigr)^{1/2}\leq \frac{1}{\sqrt{N}}\Bigl(\sum_{\ell=1}^\kappa\|\tau^\ell-\hat\tau^\ell\|^2\Bigr)^{1/2}.
	\end{align*}
	Note that
	from \eqref{sec:conop:eq2}, for any $|\lambda|\leq 1,$
	$$
	\frac{\|x\|^2}{4}-\lambda f(x)
	$$
	is a convex function on $\mathbb{R}^\kappa$ so that for any $x,y\in \mathbb{R}^\kappa,$
	\begin{align*}
		\nonumber	&\frac{1}{2}\Bigl(\frac{1}{2}\|x\|^2 - \lambda f(x)\Bigr)+	\frac{1}{2}\Bigl(\frac{1}{2}\|y\|^2 - \lambda f(y)\Bigr)-\Bigl(\frac{1}{2}\Bigl\|\frac{x+y}{2}\Bigr\|^2 - \lambda f\Bigl(\frac{x+y}{2}\Bigr)\Bigr)\\
		\nonumber	& =\frac{1}{2}\Bigl(\frac{1}{4}\|x\|^2 -\lambda f(x)\Bigr)+	\frac{1}{2}\Bigl(\frac{1}{4}\|y\|^2 - \lambda f(y)\Bigr)\\
		&\qquad-\Bigl(\frac{1}{4}\Bigl\|\frac{x+y}{2}\Bigr\|^2 - \lambda f\Bigl(\frac{x+y}{2}\Bigr)\Bigr)+\frac{1}{4}\Bigl\|\frac{x-y}{2}\Bigr\|^2\\
		%\label{add:eq2}	
		&\geq\frac{1}{4}\Bigl\|\frac{x-y}{2}\Bigr\|^2.
	\end{align*}
	This together with the fact that $-\sum_{\ell=1}^\kappa X_N(\tau^\ell)$ is concave implies that the measure $\la \cdot\ra_\lambda^c$ is a strongly log-concave measure and consequently, from the Brascamp-Lieb inequality \cite[Theorem 3.1.4]{Tal111}, the assertion follows.
\end{proof}

We are ready to establish the proof of Proposition \ref{sec:conop:prop1}. For $N\geq 1,$ let $\lambda=1/N^{1/4}.$ 
Write
\begin{align*}
	\e	\bigl\la \bigl|Q-\e\bigl\la Q\bigr\ra^c\bigr|\bigr\ra^c&\leq 	\e\bigl\la \bigl|Q-\bigl\la Q\bigr\ra^c\bigr|\bigr\ra^c+	\e\bigl\la \bigl|\bigl\la Q\bigr\ra^c-\e\bigl\la Q\bigr\ra^c\bigr|\bigr\ra^c.
\end{align*}
Note that from Lemma \ref{sec:conop:eq1}, we readily see that $\e\bigl\la \bigl|Q-\bigl\la Q\bigr\ra^c\bigr|\bigr\ra^c\leq K/\sqrt{N}.$ To handle the second term, using the convexity of $F_N^c(\lambda)$ and noting that ${F_N^c}'(0)=\la Q\ra^c,$ we can bound (see, for example, \cite[Lemma 3.2]{Pan13}),
\begin{align*}
	\e \bigl|\la Q\ra^c-\e \la Q\ra^c\bigr|&\leq \frac{1}{\lambda}\bigl(\e|F_N^c(\lambda)-\e F_N^c(\lambda))|+\e|F_N^c(0)-\e F_N^c(0)|\\
	&\qquad +\e|F_N^c(-\lambda)-\e F_N^c(-\lambda)|\bigr)+ \E \la Q \ra_\lambda^c - \E \la Q\ra_{-\lambda}^c.
\end{align*}
Here, by Lemma \ref{add:lem4},
\begin{align*}
	\la Q\ra_\lambda^c-\la Q\ra_{-\lambda}^c&=N\int_{-\lambda}^\lambda \bigl\la Q\bigl(Q-\bigl\la Q\bigr\ra_t^c\bigr)\bigl\ra_t^cdt =N\int_{-\lambda}^\lambda \bigl\la \bigl(Q-\bigl\la Q\bigr\ra_t^c\bigr)\bigl(Q-\bigl\la Q\bigr\ra_t^c\bigr)\bigr\ra_t^cdt\notag\\
	&=N\int_{-\lambda}^\lambda \bigl\la \bigl(Q-\bigl\la Q\bigr\ra_t^c\bigr)^2\bigr\ra_t^cdt \notag\leq 2\lambda K.
\end{align*}
Putting these and using the concentration in Proposition \ref{add:prop1} together yields
\begin{align*}
	%\label{add:prop1:proof:eq3}
	\e \bigl|\la Q\ra^c -\e\la Q\ra^c\bigr|&\leq \frac{K'}{\lambda \sqrt{N}}+2\lambda K=\frac{K''}{N^{1/4}}.
\end{align*}
This completes our proof.

\section{Independence of local magnetizations}

Throughout this entire section, we still assume that $\mathcal{D}$ is bounded. Let $n\geq 1$ be fixed. Assume that $(Z_1,Z_2,\ldots,Z_n)$ is a weak limit of $$\bigl(\bigl\la (\sigma_1-\la\sigma_1\ra)^2\bigr\ra,\ldots,\bigl\la (\sigma_n-\la \sigma_n\ra)^2\bigr\ra\bigr)_{N\geq 1}=\bigl(A_{11}^{-1},\ldots,A_{nn}^{-1}\bigr)_{N\geq 1}$$ along a subsequence. Note the existence of this weak limit is ensured by the fact that $0\leq A_{11}^{-1},\ldots,A_{nn}^{-1}\leq 1$.

\begin{prop}\label{add:prop:indep}
	$(Z_1,\ldots,Z_n)$ are independent and identically distributed.
\end{prop}

\begin{proof} By symmetry in the spins, $Z_1,\ldots,Z_n$ are obviously identically distributed. It remains to show that they are independent. Recall from the beginning of Section \ref{sec:conop}, we can write
	\begin{align*}
		\bigl\la (\sigma_i-\la\sigma_i\ra)^2\bigr\ra=\bigl\la |\tau_i|^2\bigr\ra^c.
	\end{align*}
	By a diagonalization process, we pass to a subsequence along which for any integer $r\geq 0,$ if $\tau_{r,i}:=(\tau_i\wedge r)\vee (-r)$, then $(\la  |\tau_{r,1}|^2\ra^c,\la |\tau_{r,2}|^2\ra^c,\ldots,\la |\tau_{r,n}|^2\ra^c)$ converges to some $(Z_{1,r},Z_{2,r},\ldots,Z_{n,r})$ weakly. For notational clarity, we shall assume that these convergences are valid in $N.$
	
	We claim that $(Z_{1,r},\ldots,Z_{n,r})$ are independent. For any $\ell\geq 1,$ denote $\tau_{r,i}^\ell=(\tau_i^\ell \wedge r)\vee (-r)$, where $\tau^\ell$ are i.i.d. samples from $\la \cdot\ra^c$. For integers $a_1,\ldots,a_n\geq 0$, let $I_s$ be the collection of integers in the interval $\bigl(\sum_{i=1}^{s-1}a_i,\sum_{i=1}^{s}a_i\bigr].$ Denote $b=\sum_{i=1}^na_i.$ Set \begin{align*}
		Q_i=	Q_i(\tau^1,\ldots,\tau^b)&=\frac{1}{N}\sum_{j=1}^N\phi_i(\tau_j^1,\ldots,\tau_j^b),
	\end{align*}
	where $\phi_i(x^1,\ldots,x^b):=\prod_{\ell\in I_i}|(x^\ell\wedge r)\vee  (-r)|^2$ for $x^1,\ldots,x^b\in \mathbb{R}$ and we adapt the tradition that $\prod_{\ell\in \emptyset}|(x^\ell\wedge r)\vee  (-r)|^2=1.$
	Note that since $\phi_i$ is a constant as long as $x^1,\ldots,x^b$ are all outside the interval $[-r,r],$ for any given $\varepsilon>0$, we can approximate $\phi_i$ uniformly, $\|\phi_i-\hat\phi_i\|_\infty<\varepsilon,$ by a smooth function $\hat\phi_i$, which is a constant whenever $x^1,\ldots,x^b$ are all outside the interval $[-r-1,r+1].$ This implies that $\hat\phi_i\in [-C_i,C_i]$, $\hat\phi_i$ is $C_i$-Lipschitz function,  and $-C_iI/4\leq \nabla^2\hat\phi_i\leq C_iI/4$ for some large constant $C_i.$ From these, $\hat\phi_i/C_i\in [-1,1]$ and it satisfies \eqref{sec:conop:eq1} and \eqref{sec:conop:eq2}. Consequently, from Proposition \ref{sec:conop:prop1}, if 
	$$
	\hat Q_i:=\frac{1}{N}\sum_{j=1}^N\hat\phi_i(\tau_j^1,\ldots,\tau_j^b),
	$$
	then
	\begin{align*}
		\e\bigl\la\bigl|\hat Q_i-\e\bigl\la \hat Q_i\bigr\ra^c\bigr|\bigr\ra^c&\leq \frac{C_iK}{N^{1/4}}
	\end{align*}
	for some constant $K$ depending on $\alpha,\beta,p,b.$ As a result,
	\begin{align*}
		\e\bigl\la\bigl|Q_i-\e\bigl\la Q_i\bigr\ra^c\bigr|\bigr\ra^c&\leq \frac{C_iK}{N^{1/4}}+2\varepsilon.
	\end{align*}
	Next using the symmetry among replicas and spins and this inequality, for fixed $\varepsilon>0$ and any large enough $N$,
	\begin{align*}
		\e \prod_{i=1}^n\bigl(\la |\tau_{r,i}|^2\ra^c\bigr)^{a_i}&=\e\Bigl\la \prod_{i=1}^n\prod_{\ell\in I_i} |\tau_{r,i}^\ell|^{2}\Bigr\ra^c=\e \Bigl\la \prod_{i=1}^n Q_i\Big\ra^c+o(1)
		\\
		&=\prod_{i=1}^n\e \bigl\la  Q_i\big\ra^c+o(1)+O(\varepsilon)=	\prod_{i=1}^n\e \bigl(\la |\tau_{r,i}|^2\ra^c\bigr)^{a_i}+o(1)+O(\varepsilon).
	\end{align*}
	Note that the first error $o(1)$ accounts for those terms in the expansion $\prod_{i=1}^n Q_i$ whose spin indices are not distinct and that the other two $o(1)$ errors are the same as the first. It follows that by using the dominated convergence theorem in the limit $N\to\infty$ and then sending $\varepsilon\downarrow 0$,
	\begin{align*}
		\e \prod_{i=1}^nZ_{i,r}^{a_i}&= \prod_{i=1}^n\e Z_{i,r}^{a_i}.
	\end{align*}
	Since this is valid for any integers $a_1,\ldots,a_n\geq 0$ and $Z_{1,r},\ldots,Z_{n,r}$ are bounded, we conclude that $Z_{1,r},\ldots,Z_{n,r}$ are independent of each other.
	
	The remaining step is to show that $Z_1,\ldots,Z_n$ are independent.
	Let $f_1,\ldots,f_n$ be bounded and Lipschitz functions on $\mathbb{R}$. Assume that their supremum norms and Lipschitz  constants are all bounded by $K_0$.  Write
	\begin{align*}
		\la \tau_i\ra^c -\la \tau_{r,i}\ra^c &=\la (\tau_i-r)\mathbb{I}(\tau_i\geq r)\ra^c +\la (\tau_i+r)\mathbb{I}(\tau_i\leq -r)\ra^c\\
		&=\la \mbox{sign}(\tau_i)(|\tau_i|-r)\mathbb{I}(|\tau_i|\geq r)\ra^c, 
	\end{align*}
	which implies that by the Cauchy-Schwarz inequality,
	\begin{align*}
		|\la \tau_i\ra^c -\la \tau_{r,i}\ra^c|\leq\sqrt{\la |\tau_i|^2\ra^c}\sqrt{\la \mathbb{I}(|\tau_i|\geq r)\ra^c }.
	\end{align*}
	Hence, 
	\begin{align*}
		\Bigl| \prod_{i=1}^nf_i(\la \tau_i\ra^c )-\prod_{i=1}^nf_i(\la \tau_{r,i}\ra^c )\Bigr|
		\leq K_0^n\sum_{i=1}^n\sqrt{\la |\tau_i|^2\ra^c }\sqrt{\la \mathbb{I}(|\tau_i|\geq r)\ra^c }.
	\end{align*}
	Consequently, by the symmetry among the spins, the Cauchy-Schwarz inequality, and then the Markov inequality,
	\begin{align*}
		\Bigl| \e\prod_{i=1}^nf_i(\la \tau_i\ra^c )-\e\prod_{i=1}^nf_i(\la \tau_{r,i}\ra^c )\Bigr|&\leq nK_0^n \bigl(\e \la |\tau_1|^2\ra^c \bigr)^{1/2}\bigl(\e\la \mathbb{I}(|\tau_1|\geq r)\ra^c\bigr)^{1/2}\\
		&\leq \frac{nK_0^n\e\la|\tau_1|^2\ra^c}{r}\leq  \frac{nK_0^n}{r},
	\end{align*}
	where the last inequality used the bound $\e\la|\tau_1|^2\ra^c=\e A_{11}^{-1}\leq 1$ since $A_{11}^{-1}\leq 1.$
	Sending $N$ to infinity and using the independence  of $Z_{1,r},\ldots,Z_{n,r}$, this inequality readily implies that
	$$
	\e \prod_{i=1}^nf_i(Z_i)=\prod_{i=1}^n\e f_i(Z_i).$$ Since this equation holds for all bounded Lipschitz functions $f_i,$ it follows that $Z_1,\ldots,Z_n$ are independent of each other and this completes our proof.
\end{proof}

\section{Convergence of the spin variance}\label{conv:spinvariance}

Let $T$ be the operator defined in \eqref{fixedpointeq}. For the reader's convenience, we recall that for any $\mu\in \mathcal{P}([0,1])$, $T(\mu)$ is defined as the distribution of 
\begin{align}\label{conv:spinvariance:eq1}
	\Bigl(1+\sum_{k=1}^{R}\frac{2\beta\zeta_k^2}{1+2\beta \sum_{r=1}^{p-1}X_{k,i}\xi_{k,r}^2}\Bigr)^{-1},
\end{align} 	
where $(\zeta_k)_{k\geq 1},(\xi_{k,r})_{k,r\geq 1}\stackrel{i.i.d.}{\sim} \mathcal{D}$, $(X_{k,r})_{k,r\geq 1}\stackrel{i.i.d.}{\sim}\mu$, $R$ is Poisson$(\alpha p)$, and these are all independent of each other. This section is devoted to establishing the one-dimensional case of Theorem \ref{thm:fixedpoint} assuming that $\mathcal{D}$ is bounded.

\begin{prop}\label{add:prop3}
	Assume that $\mathcal{D}$ is bounded. The law of the random variable $(A_N^{-1})_{NN}$ converges weakly to the unique solution of the distributional equation $T(\mu)=\mu.$
\end{prop}

The remaining of this section will establish this proposition. Let $Q$ and $R$ be two independent Poisson random variables  with mean $\alpha(N-p)$ and $\alpha p$, respectively. 
Consider
\begin{align*}
	(\hat I(k,1), \ldots \hat I(k,p))_{k \geq 1}&\,\,\mbox{i.i.d. uniform on $\{(i_1, \ldots, i_p): 1\leq i_r \leq N-1, \text{ all distinct}\}$},\\
	(\bar I(k,1), \ldots, \bar I(k,p-1))_{k \geq 1}&\,\,\mbox{i.i.d. uniform on $\{(i_1, \ldots, i_{p-1}): 1\leq i_r\leq N-1, \text{ all distinct}\}$}.
\end{align*}
In addition, let $(\hat \xi_{k,i})_{k \geq 1,1 \leq i \leq N-1}$, $(\zeta_k)_{k\geq 1},$ and $(\xi_{k,i})_{k\geq 1,1\leq i\leq N - 1}$ be i.i.d. sampled from $\mathcal{D}$. Assume that these are all independent of each other. Recall that $(e_i)_{1\leq i\leq N}$ is the standard basis of $\mathbb{R}^N$. Set the column vectors
\begin{align*}
	u_k&=\sum_{r=1}^p\hat \xi_{k,\hat I(k,r)}e_{\hat I(k,r)},\\
	w_k&=\sum_{r=1}^{p-1}\xi_{k,\bar I(k,r)} e_{\bar I(k,r)}+\zeta_k e_N.
\end{align*}
Using the thinning property of the Poisson random variable, we can write  
\begin{align}\label{add:eqdecomp}
	A\stackrel{d}{=} I + \tb\sum_{k \leq Q} u_ku_k^T + \tb\sum_{k \leq R}w_kw_k^T = B +  \tb\sum_{k \leq R}w_kw_k^T.
\end{align}
In words, we decompose $A$ into two components in distribution. The first, $B$, is a block matrix with the principal $(N-1)\times (N-1)$ block recording the entries of $A$ at indices belonging to the set $\{1, \ldots, N-1\}$ and it satisfies $B_{Ni} = B_{iN} = 0$ for all $1\leq i \leq N-1$ and $B_{NN} = 1$. The term $\sum_{k \leq R}w_kw_k^T$  accounts for the entries of $A$ whose indices are connected to the vertex $N.$ From \eqref{add:eqdecomp}, for the rest of this section,  we assume that
\begin{align}\label{add:eqdecomp:eq1}
	A= B +  \tb\sum_{k \leq R}w_kw_k^T.
\end{align}

\subsection{Some preliminary estimates}
Define the matrix $E = (E_{kl })_{1\leq k,l \leq R}$ as \begin{align*}
	E_{kl } = \begin{cases}
		\tb\sum_{1\leq r,s\leq p-1} \xi_{k, \bar I(k,r)}\xi_{l ,\bar I(l ,s)} B^{-1}_{\bar I(k,r), \bar I(l ,s)},  & \text{ if } k\neq l ,\\
		\tb\sum_{1\leq r\neq s \leq p-1}\xi_{k, \bar I(k,r)}\xi_{k,\bar I(k,s)} B^{-1}_{\bar I(k,r), \bar I(k,s)}, & \text{ if } k = l
	\end{cases}
\end{align*}
and let $\zeta = (\zeta_1,\ldots, \zeta_R)^T$. 

\begin{lemma}\label{add:lem7}
	There exists a constant $K$ independent of $N$ such that \begin{align*}
		\Bigl|A^{-1}_{NN} - \Bigl(1 + \sum_{k=1}^R \frac{\tb\zeta_k^2}{1 + \tb\sum_{r=1}^{p-1}\xi_{k,\bar I(k,r)}^2 B^{-1}_{\bar I(k,r),\bar I(k,r)}}\Bigr)^{-1}\Bigr| \leq K\|\zeta\|^2\|E\|.
	\end{align*}
\end{lemma}

\begin{proof}
	Let $W \in \R^{N\times R}$ be the matrix that records the vectors $(\sqrt{\tb}w_k)_{1\leq k\leq R}$ along the columns. Let $D \in \R^{R\times R}$ be a diagonal matrix whose $k$th diagonal entry is $1 + \tb\sum_{r=1}^{p-1}\xi_{k,\bar I(k,r)}^2 B^{-1}_{\bar I(k,r),\bar I(k,r)}$. From the identity \eqref{add:eqdecomp:eq1} and the Woodbury matrix identity, we write
	\begin{align*}
		A^{-1}_{NN} & = B^{-1}_{NN} - (B^{-1}W(I + W^TB^{-1}W)^{-1}W^TB^{-1})_{NN}.
	\end{align*}
	Noting the block structure of $B$, we obtain \begin{align*}
		A^{-1}_{NN}	& = 1 - \tb\zeta^T(I + W^TB^{-1}W)^{-1}\zeta\\
		& = 1 - \tb\zeta^T(\tb\zeta\zeta^T + D + E)^{-1}\zeta.
	\end{align*}
	The second line above follows from the definition of $D$ and $E$. Indeed, for $1\leq k,l \leq R$,\begin{align*}
	(I + W^TB^{-1}W)_{kl} & = \delta_{kl}+ \sum_{i,j=1}^N W_{ik}W_{jl}B^{-1}_{ij} = \delta_{kl}+ 2\beta\zeta_k\zeta_{l}+ \sum_{i,j=1}^{N-1}W_{ik}W_{jl}B^{-1}_{ij} \\
	&  = \delta_{kl}  +\tb\zeta_k\zeta_{l}+ 2\beta\sum_{1\leq r,s\leq p-1} \xi_{k, \bar I(k,r)}\xi_{l, \bar I(l,s)}B^{-1}_{\bar I(k,r)\bar I(l,s)}\\
	&  = \tb\zeta_k\zeta_l+D_{kl} + E_{kl},
\end{align*} where $\delta_{kl}$ is the Kronecker delta between $k$ and $l$ and  in the second equality we used that $B^{-1}_{Nj} = B^{-1}_{jN} = 0$ for $j \neq N$ and $B^{-1}_{NN} = 1$. Using the resolvent identity followed by the Sherman-Morrison formula, we further obtain \begin{align*}
		A^{-1}_{NN}	& = 1 - \tb\zeta^T(\tb\zeta\zeta^T + D)^{-1}\zeta +\Delta\\
		& = 1 - \tb\zeta^T\left(D^{-1} - \tb\frac{D^{-1}\zeta\zeta^TD^{-1}}{1 + \tb\zeta^T D^{-1}\zeta}\right)\zeta + \Delta\\
		& = \frac{1}{1 + \tb\zeta^TD^{-1}\zeta} + \Delta\\
		& = \left(1 + \sum_{k=1}^R \frac{\tb\zeta_k^2}{1 + \tb\sum_{r=1}^{p-1}\xi_{k,\bar I(k,r)}^2 B^{-1}_{\bar I(k,r),\bar I(k,r)}}\right)^{-1} +\Delta,
	\end{align*}
	where
	\begin{align*}
		\Delta&= \tb \zeta^T(\tb\zeta\zeta^T + D + E)^{-1}E(\tb\zeta\zeta^T + D)^{-1}\zeta.
	\end{align*}
	Note that the last term above can be bounded as 
	\begin{align*}
		|\Delta|&\leq K\|\zeta\|^2 \|(\tb\zeta\zeta^T + D + E)^{-1}\|\|E\|\|(\tb\zeta\zeta^T + D)^{-1}\|.
	\end{align*}
	The result follows since $\|(\tb\zeta\zeta^T + D + E)^{-1}\|\leq1$ and $ \|(\tb\zeta\zeta^T + D)^{-1}\| \leq 1$.
\end{proof}

The estimate in Lemma \ref{add:lem7} involves the main diagonal of $B^{-1}$, which is of constant order. The following lemma shows that the entries of $B^{-1}$ will be of order $O(1/\sqrt{N})$ if the random indices are not identically the same.

\begin{lemma}\label{add:lem8} 
	For any $N \geq 3$, there exists a constant $K$ independent of $N$ such that \begin{align*}
		\max_{\substack{{1\leq k \leq l}\\{1\leq r<s\leq p-1}}} \E\bigl[\bigl(B^{-1}_{\bar I(k,r), \bar I(k,s)}\bigr)^2 \big| R = l\bigr] \leq \frac{K}{N},& \quad \forall\; l \geq 1,\\
		\max_{\substack{{1\leq k <t\leq l}\\{1\leq r\leq s\leq p-1}}} \E\bigl[\bigl(B^{-1}_{\bar I(k,r), \bar I(t,s)}\bigr)^2 \big| R = l\bigr] \leq \frac{K}{N}, & \quad\forall\; l \geq 2.
	\end{align*}
\end{lemma}

\begin{proof}
	Note that for $1\leq a \leq N-1$,\begin{align*}
		1 \geq (B^{-2})_{aa} = \sum_{b = 1}^{N-1}(B^{-1}_{ab})^2 \geq \sum_{\substack{1\leq b\leq N-1\\ b \neq a}} (B^{-1}_{ab})^2.
	\end{align*}
	By symmetry between the sites, for $1\leq b,b' \leq N-1$ and $b, b' \neq a$, we have\begin{align*}
		\E\bigl[(B^{-1}_{ab})^2 \big| R = l\bigr] = \E\big[(B^{-1}_{ab'})^2 \big| R=l\big].
	\end{align*}
	Hence, we obtain that for $1\leq b \leq N-1$ with $b \neq a$, \begin{align}\label{add:eq-13}
		\E\big[(B^{-1}_{ab})^2\big| R=l\big] \leq \frac{1}{N-2}.
	\end{align}
	Now, for any $1\leq k\leq l$ and $1\leq r<s\leq p-1$, we have \begin{align*}
		\E\bigl[\bigl(B^{-1}_{\bar I(k,r), \bar I(k,s)}\bigr)^2 \big| R = l\bigr]& = 2\sum_{1\leq a<b \leq N-1} \E\bigl[(B^{-1}_{ab})^2\big| R=l\bigr] \P\bigl(\bar I(k,r) = a, \bar I(k,s) = b\bigr)\\
		& \leq 2\cdot\frac{(N-1)(N-2)}{2}\cdot\frac{1}{N-2}\cdot\frac{2}{(N-1)(N-2)} \leq \frac{K}{N}.
	\end{align*}
	To prove the second assertion, let $\mathcal{E}_l$ be the event that the sites $\{\bar{I}(k,r) : 1\leq k\leq l, 1\leq r\leq p-1\}$ are all distinct and $R=l$ and $\mathcal{E}_l'=\mathcal{E}_l^c\cap\{R=l\}$. Then \begin{align}\label{add:eq-14}
		& \P\bigl(\mathcal{E}_l'\big|R = l\bigr)=
		1 - \prod_{k=0}^{l-1}\frac{{N-1-k(p-1)\choose p-1}}{{N-1\choose p-1}} \leq \frac{K}{N}.
	\end{align}
	It follows that for any $1\leq k <t \leq l$ and $1\leq r\leq s\leq p-1$, 
	\begin{align*}
		\E\bigl[\bigl(B^{-1}_{\bar I(k,r), \bar I(t,s)}\bigr)^2 \big| R = l\bigr] &= \E\bigl[\bigl(B^{-1}_{\bar I(k,r), \bar I(t,s)}\bigr)^2\mathbb{I}(\mathcal{E}_l)\big|R=l\bigr]+
		\E\bigl[\bigl(B^{-1}_{\bar I(k,r), \bar I(t,s)}\bigr)^2\mathbb{I}(\mathcal{E}_l')\big|R=l\bigr]\\
		&\leq \E\bigl[\bigl(B^{-1}_{1,2}\bigr)^2\big|R=l\bigr]+\P\bigl(\mathcal{E}_l'\big|R=l\bigr)\leq \frac{K'}{N},
	\end{align*}
	where the first inequality used symmetry among the sites and the fact $|B_{ij}^{-1}|\leq 1$, while the second inequality holds thanks to \eqref{add:eq-13} and \eqref{add:eq-14}.
\end{proof}

\subsection{Uniqueness of the fixed point solution}

The following lemma takes care of the uniqueness of the fixed point solution of $T$ stated in Proposition \ref{add:prop3}, where we do not need $\mathcal{D}$ to be bounded.

\begin{lemma}\label{add:lem9}
	The solution to the fixed point equation, $T(\mu)=\mu,$ is unique.
\end{lemma}

\begin{proof} 
	Let $q\geq 1.$ Denote by $\mathcal{P}_q(\mathbb{R}_+)$ the collection of all $\nu\in \mathcal{P}(\mathbb{R}_+)$ with $\int x^q\nu(dx)<\infty.$ We equip this space with the Wasserstein $q$-distance defined as
	\begin{align*}
		W_q(\nu_1,\nu_2)&=\inf\bigl( \e\bigl|Z_1-Z_2\bigr|^q\bigr)^{1/q}
	\end{align*}
	for any $\nu_1,\nu_2\in \mathcal{P}_q(\mathbb{R}_+)$,
	where the infimum is taken over all joint random vectors $(Z_1,Z_2)$ with $Z_1\sim \nu_1$ and $Z_2\sim \nu_2.$ 	Denote $\gamma=(2\beta)^{-1}.$ Let $\phi(x) =  - \log x$. Understanding $\phi^{-1}(\nu)$ and $\phi(\mu)$ as the push-forward measures of $\mu\in \mathcal{P}([0,1])$ and $\nu\in \mathcal{P}_q(\mathbb{R}_+)$ under $\phi$ and $\phi^{-1}$ respectively, we define a self-map on $\mathcal{P}_q(\mathbb{R}_+)$ as $\mathcal{T} = \phi\circ T\circ \phi^{-1}$, namely, for any $\nu\in \mathcal{P}_q(\mathbb{R}_+)$, $\mathcal{T}(\nu)$ is the distribution of 
	$$\log\Bigl(1 + \sum_{k=1}^R\frac{\zeta_k^2}{\gamma + \sum_{r=1}^{p-1} \xi_{k,r}^2e^{-Y_{k,r}}}\Bigr),
	$$
	where $R\sim \text{Poisson}(\alpha p)$, $(\xi_{k,r})_{k,r}$ and $(\zeta_k)_k \stackrel{i.i.d.}{\sim} \mathcal{D}$, $Y_{k,r}\stackrel{i.i.d.}{\sim}\nu$, and these are all independent of each other. To see that $\mathcal{T}$ is a self-map, note that for all $q\geq 1$ and $\nu\in \mathcal{P}_q(\mathbb{R}_+),$ if $\nu'=\mathcal{T}(\nu)$, then there exists a constant $K>0$ such that
	\begin{align}\label{add:eq-22}
		\int x^q\nu'(dx)&\leq \e\log^q\Bigl(1+\sum_{k=1}^R\frac{\zeta_k^2}{\gamma}\Bigr)\leq K \e \sum_{k=1}^R\zeta_k^2\leq K\alpha p\e\xi_1^2<\infty.
	\end{align}
	
	We claim that $\mathcal{T}$ is a contraction as long as $q$ is large enough. To this end, for $l \geq 1$, define $g_l:\R_+^l \times \R_+^{p-1} \to \R_+$ by 
	$$g_l(y): = g_l\bigl((y_{k,r})_{k\leq l, r\leq p-1}\bigr) = \log\Bigl(1 + \sum_{k=1}^l\frac{\zeta_k^2}{\gamma + \sum_{r=1}^{p-1} \xi_{k,r}^2e^{-y_{k,r}}}\Bigr).$$ A direct computation gives 
	\begin{align*}
		\sum_{k,r}|\partial_{y_{k,r}}g_r(y)| 
		& = \frac{\sum_{k=1}^l  \frac{\zeta_k^2\Delta_k}{(\gamma +\Delta_k)^2}}{1 + \sum_{k=1}^l\frac{\zeta_k^2}{\gamma +\Delta_k}}\leq  \frac{\sum_{k=1}^l  \frac{\zeta_k^2}{\gamma +\Delta_k}}{1 + \sum_{k=1}^l\frac{\zeta_k^2}{\gamma +\Delta_k} }\leq \frac{\sum_{k=1}^l  \frac{\zeta_k^2}{\gamma}}{1 + \sum_{k=1}^l\frac{\zeta_k^2}{\gamma} }= \frac{\chi_l}{\gamma + \chi_l }
	\end{align*}
	for 
	$\Delta_k:=\sum_{r=1}^{p-1}\xi_{k,r}^2e^{-y_{k,r}}$ and 
	$\chi_l:=\sum_{k=1}^l \zeta_k^2.$
	Now for  two probability measures $\nu_1,\nu_2\in \mathcal{P}_q(\mathbb{R}_+)$,
	let $\nu_*$ be an optimal coupling between $\nu_1$ and $\nu_2$ under the Wasserstein $q$-distance, namely,
	$$
	W_q^q(\nu_1,\nu_2)=\int |x_1-x_2|^qd\nu_*.
	$$ 
	For each $k,r$, let $(Y_{k,r}, Y'_{k,r})$ be sampled independently from $\nu_*$ and be independent of other randomness. Let $Y: = ((Y_{k,r})_{k\geq 1, r\leq p-1})$ and $Y': = ((Y'_{k,r})_{k\geq 1, r\leq p-1})$. Then, by the mean-value theorem and the inequality above, we have for every $q\geq 1$,\begin{align*}
		\E |g_l(Y) - g_l(Y')|^q & \leq \E\sup_y\|\nabla g_l(y)\|_1^q \max_{k\leq l, r\leq p-1} \bigl|Y_{k,r} - Y'_{k,r}\bigr|^q \\
		& \leq \Bigl(\E  \sup_y \|\nabla g_l(y)\|_1^q\Bigr) l(p-1) W_q^q(\nu_1, \nu_2)\\
		& \leq \Bigl(\E \Bigl(\frac{\chi_l}{\gamma + \chi_l}\Bigr)^q \Bigr) l(p-1) W_q^q(\nu_1, \nu_2).
	\end{align*}
	Hence, we arrive at \begin{align*}
		W_q^q(\mathcal{T}(\nu_1), \mathcal{T}(\nu_2)) \leq \E  \Bigl[\Bigl(\frac{\chi_R}{\gamma + \chi_R}\Bigr)^q R(p-1)\Bigr] W_q^q(\nu_1, \nu_2).
	\end{align*}
	Here, choosing $q$ sufficiently large, the expectation on the right-hand side is strictly less than $1.$ Indeed, this can be seen by noting that  
	$$
	\Bigl(\frac{\chi_R}{\gamma + \chi_R}\Bigr)^q R(p-1)\leq R(p-1)
	$$
	and the left-hand side converges to zero a.s. as $q\uparrow \infty$ and applying the dominated convergence theorem.
	This completes the proof of our claim.
	
	Now we turn to the proof of the uniqueness of the fixed point of $T$. Assume that $\mu_1$ and $\mu_2$ are two distinct fixed points of $T.$ Note that from the definition of $T$ in \eqref{conv:spinvariance:eq1}, both $\mu_1$ and $\mu_2$ can not charge positive masses at $0.$ From this, the distributions $\nu_1$ and $\nu_2$ of $-\log X_1$ and $-\log X_2$ for $X_1\sim \mu_1$ and $X_2\sim \mu_2$ are probability distributions on $\mathbb{R}_+.$ Furthermore, in a similar manner as \eqref{add:eq-22}, $\e|\log X_1|^q<\infty$ and $\e|\log X_2|^q<\infty,$ which implies that $\nu_1,\nu_2\in \mathcal{P}_q(\mathbb{R}_+)$. Since now $\nu_1$ and $\nu_2$ are two distinct fixed points of $\mathcal{T}$, this contradicts the contractivity of $\mathcal{T}.$ Hence, the fixed point of $T$ must be unique. 
\end{proof}

\subsection{Proof of Proposition \ref{add:prop3}}

For convenience,  with a slight abuse of notation,  we shall use $T(X)$ to stand for the random variable whose law is given by the application of $T$ to the law of $X.$ In this notation, $T(\mu)=\mu$ is equivalent to $T(X)\stackrel{d}{=}X$ for $X\sim \mu.$

Let $(\bar J(k,1), \ldots, \bar J(k,p-1))_{k \geq 1}$ be an independent copy of $\{\bar I(k,1),\ldots,\bar I(k,p-1)\}_{k\geq 1}$ and $(\xi_{k,i}')_{k\geq 1,i\geq 1}$ and $(\xi'_k)_{k\geq 1}$ be i.i.d.\ copies of $\mathcal{D},$ and $R'$ be Poisson$(\alpha p).$  Assume that these are all independent of each other. Recall that $\{e_1,\ldots,e_N\}$ is the standard basis for $\mathbb{R}^N.$ Let 
$$
w'_k=\sum_{r=1}^{p-1}\xi'_{k,\bar J(k,r)} e_{\bar J(k,r)}+\zeta_k'e_N,
$$
Recall the decomposition of $A$ in \eqref{add:eqdecomp:eq1}.  Let
\begin{align*}
	A' =  B + \tb\sum_{k=1}^{R'} w_k'w_k'^T.
\end{align*}
Note that $A' \stackrel{d}{=} A$. Let $f:\R \to \R$ be a bounded Lipschitz function. Set
\begin{align*}
	U&=\Bigl(1 + \sum_{k=1}^R\frac{\tb\zeta_k^2}{1+\tb\sum_{r=1}^{p-1}\xi_{k,\bar I(k,r)}^2 B^{-1}_{\bar I(k,r), \bar I(k,r)}}\Bigr)^{-1},\\
	V&=\Bigl(1 + \sum_{k=1}^R\frac{\tb\zeta_k^2}{1+\tb\sum_{r=1}^{p-1}\xi_{k,\bar I(k,r)}^2 A'^{-1}_{\bar I(k,r), \bar I(k,r)}}\Bigr)^{-1}.
\end{align*}
We have \begin{align} 
	\begin{split}\label{add:prop3:eqn1}
		\bigl |\E f\bigl(A^{-1}_{NN}\bigr) - \E f\bigl(T(A'^{-1}_{NN})\bigr)\bigr|&\leq \|f'\|_\infty\E \bigl| A^{-1}_{NN} - U\bigr|+ \|f'\|_\infty\E \bigl|U  - V\bigr| \\
		&\qquad+ \bigl|\E f(V) - \E f(T(A'^{-1}_{NN}))\bigr|.
	\end{split}
\end{align}
We will now bound the three terms in \eqref{add:prop3:eqn1}. 

\smallskip

{\noindent \bf First term:} Recall the definition of $\mathcal{E}_l$ and $\mathcal{E}_l'$ from the proof of Lemma \ref{add:lem8}. From Lemma \ref{add:lem7}, $|A^{-1}_{NN}|\leq 1$, $|U|\leq 1,$ $\mathcal{E}_l\subseteq \{R=l\}$, and $\P(\mathcal{E}_l')\leq KN^{-1}\P(R=l)$ (thanks to \eqref{add:eq-14}), we can bound  \begin{align}\label{add:prop3:eqn2}
	\E \bigl| A^{-1}_{NN} - U\bigr|
	&=\sum_{l=0}^\infty\E \bigl[\bigl| A^{-1}_{NN} - U\bigr|\big|\mathcal{E}_l\bigr]\P(\mathcal{E}_l)+\sum_{l=0}^\infty\E \bigl[\bigl| A^{-1}_{NN} - U\bigr|\big|\mathcal{E}_l'\bigr]\P(\mathcal{E}_l')\notag\\
	& \leq K\sum_{l=0}^\infty \E \bigl[\|\zeta\|^2\|E\|\big| \mathcal{E}_l \bigr]\P(R=l) +\frac{2K}{N}\sum_{l=0}^\infty \P(R = l)  \notag\\
	&\leq  K'\sum_{l=0}^\infty \E \bigl[\|E\|\big| \mathcal{E}_l \bigr]\P(R=l) + \frac{2K}{N},
\end{align}
where the second inequality used the boundedness of $\mathcal{D}$. Using the definition of $E$ and letting $\|E\|_F$ to denote the Fr\"obenius norm of $E$, from the estimates in Lemma \ref{add:lem8}, we have \begin{align*}
	\E \bigl[\|E\|^2\big| \mathcal{E}_l\bigr] & \leq \E \bigl[\|E\|_F^2 \big| \mathcal{E}_l\bigr] \\
	&  = 4\beta^2\sum_{k\leq l}\E\Bigl[\Bigl(\sum_{1\leq r \neq s \leq p-1}\xi_{k, \bar I(k,r)}\xi_{k,\bar I(k,s)} B^{-1}_{\bar I(k,r), \bar I(k,s)}\Bigr)^2 \Big| \mathcal{E}_l\Bigr]\\
	& \qquad + 8\beta^2 \sum_{1\leq k < t \leq l}\E\Bigl[\Bigl(\sum_{1\leq r,s \leq p-1}\xi_{k, \bar I(k,r)}\xi_{k,\bar I(t,s)} B^{-1}_{\bar I(k,r), \bar I(t,s)}\Bigr)^2 \Big| \mathcal{E}_l\Bigr] \\
	&\leq \frac{Kl}{N} + \frac{Kl^2}{N}.
\end{align*}
Thus, from \eqref{add:prop3:eqn2} and noting that $R\sim\mbox{Poisson}(\alpha p),$ we obtain that \begin{align}\label{add:prop3:eqn3}
	\E \bigl| A^{-1}_{NN} - U\bigr| \leq \frac{K''}{\sqrt{N}}.
\end{align}

{\noindent \bf Second term:} Note that \begin{align*}
	& \E |U-V|\\
	& \leq \tb\E\Bigl|\sum_{k=1}^R  \zeta_k^2 \Bigl(\frac{1}{1+\tb\sum_{r=1}^{p-1}\xi_{k,\bar I(k,r)}^2 B^{-1}_{\bar I(k,r), \bar I(k,r)}} - \frac{1}{1+\tb\sum_{r=1}^{p-1}\xi_{k,\bar I(k,r)}^2 A'^{-1}_{\bar I(k,r), \bar I(k,r)}}\Bigr)\Bigr|\\
	& \leq 4\beta^2 \E \Bigl| \sum_{k=1}^R \sum_{r=1}^{p-1}\zeta_k^2\xi_{k, \bar I(k,r)}^2 \Bigl(B^{-1}_{\bar I(k,r), \bar I(k,r)} - A'^{-1}_{\bar I(k,r), \bar I(k,r)}\Bigr)\Bigr|\\
	& \leq K\E\Bigl|B^{-1}_{1,1} - A'^{-1}_{1,1}\Bigr|,
\end{align*}
where for the last line, we used the independence of the quantities $R$, the boundedness of $\zeta_k$ and $\xi_{k,i}$, and the symmetry between the sites $1, \ldots, N-1$. To proceed, we can use the resolvent identity to obtain
\begin{align*}
	\E\left|B^{-1}_{11} - A'^{-1}_{11}\right|& = \tb\E\Bigl|\sum_{k=1}^{R'}\sum_{r,s=1}^{p-1}\xi'_{k,\bar J(k,r)}\xi'_{k,\bar J(k,s)}B^{-1}_{1,\bar J(k,r)}A'^{-1}_{1,\bar J(k,s)}\Bigr|\\
	& \leq K\E\bigl|B^{-1}_{1,\bar J(1,1)}\bigr|\leq \frac{K'}{\sqrt{N}},
\end{align*}
where the first inequality used $|A_{ij}'^{-1}|\leq 1$ and the boundedness of $\mathcal{D}$ and the second inequality used \eqref{add:eq-13}. Therefore,
\begin{align} \label{add:prop3:eqn5}
	& \E |U-V| \leq \frac{K''}{\sqrt{N}}.
\end{align}
{\noindent \bf Third term:} Write
\begin{align*}
	\e f(V)&=\sum_{l=0}^\infty\e\bigl[f(V)\big|\mathcal{E}_l\bigr]\P(\mathcal{E}_l)+\sum_{l=0}^\infty\e\bigl[f(V)\big|\mathcal{E}_l'\bigr]\P(\mathcal{E}_l').
\end{align*}
Here, on one hand, from \eqref{add:eq-14}, $\P(\mathcal{E}_l')\leq K/N.$ On the other hand, note that on the event $\mathcal{E}_l$, the indices $\bar I(k,r)$ for $1\leq k\leq l$ and $1\leq r\leq p-1$ are all distinct. Let $L(k,r) = (k-1)(p-1) + r$ for $k \geq 1$ and $1\leq r \leq p-1$. Since $A'$ is independence of $\{\bar I(k,r):1\leq k\leq l,1\leq r\leq p-1\}$, $(\xi_{k,i})_{k,i\geq 1}$, and $(\zeta_k)_{k\geq 1}$, using the symmetry between the diagonal entries of $A'^{-1}$ leads to
\begin{align*}
	\e\bigl[ f(V)\big|\mathcal{E}_l\bigr]&=\e f\Bigl(\Bigl(1 + \sum_{k=1}^l\frac{\tb\zeta_k^2}{1+\tb\sum_{r=1}^{p-1}\xi_{k,L(k,r)}^2 A'^{-1}_{L(k,r), L(k,r)}}\Bigr)^{-1}\Bigr)\\
	&=\e f\Bigr(\Bigl(1 + \sum_{k=1}^l\frac{\tb\zeta_k^2}{1+\tb\sum_{r=1}^{p-1}\xi_{k,r}^2 A'^{-1}_{L(k,r), L(k,r)}}\Bigr)^{-1}\Bigr).
\end{align*}
Consequently, using these and noting that $\mathcal{E}_l\cup\mathcal{E}_l'=\{R=l\}$ yields 
\begin{align*}
	\Bigl|\e f(V)-\e f\Bigr(\Bigl(1 + \sum_{k=1}^R\frac{\tb\zeta_k^2}{1+\tb\sum_{r=1}^{p-1}\xi_{k,r}^2 A'^{-1}_{L(k,r), L(k,r)}}\Bigr)^{-1}\Bigr)\Bigr|\leq \frac{K'}{N}.
\end{align*}
Hence, we arrive at
\begin{align}
	&\bigl|\E f(V)- \E f\bigl(T(A'^{-1}_{NN})\bigr)\bigr|\notag\\
	& \leq \Bigl|\e f\Bigr(\Bigl(1 + \sum_{k=1}^R\frac{\tb\zeta_k^2}{1+\tb\sum_{r=1}^{p-1}\xi_{k,r}^2 A'^{-1}_{L(k,r), L(k,r)}}\Bigr)^{-1}\Bigr) -  \e f\bigl(T(A'^{-1}_{NN})\bigr) \Bigr|+ \frac{K'}{N}. \label{add:prop3:eqn7}
\end{align} 

To sum up, after combining \eqref{add:prop3:eqn1}, \eqref{add:prop3:eqn3}, \eqref{add:prop3:eqn5},  and \eqref{add:prop3:eqn7}, we obtain that for any bounded Lipschitz function $f$,
\begin{align*}
	&\bigl|\E f\bigl((A^{-1}_{N})_{N N}\bigr) - \E f\bigl(T((A'^{-1}_{N})_{NN})\bigr)\bigr| \\
	& \leq \Bigl|\e  f\Bigr(\Bigl(1 + \sum_{k=1}^R\frac{\tb\zeta_k^2}{1+\tb\sum_{r=1}^{p-1}\xi_{k,r}^2 A'^{-1}_{L(k,r), L(k,r)}}\Bigr)^{-1}\Bigr) -  \e f\bigl(T(A'^{-1}_{NN})\bigr) \Bigr|+ \frac{K''}{\sqrt{N}}.
\end{align*}
Now since $0\leq (A_N'^{-1})_{ii}\leq 1$ for all $1\leq i\leq N$ and $N\geq 1$, by a diagonalization procedure, we can assume without loss of generality that for each $l\geq 1,$ as $N$ tends to infinity,
\begin{align*}
	&\left((A'^{-1}_{N})_{11},\ldots,(A'^{-1}_{N})_{l(p-1), l(p-1)},(A_N'^{-1})_{NN}\right)\\
	 &\stackrel{d}{=}	\left((A^{-1}_{N})_{11},\ldots,(A^{-1}_{N})_{l(p-1), l(p-1)},(A_N^{-1})_{NN}\right)\\
	& \Rightarrow (X_1, \ldots, X_{l(p-1)},X)
\end{align*}
for some random vector $(X_1,\ldots,X_{l(p-1)},X).$ By symmetry among the sites and Proposition \ref{add:prop:indep},  $X_1,\ldots, X_{l(p-1)},X$ are independent and identically distributed.  Using the continuous mapping theorem, it follows that \begin{align*}
	&\lim_{N\to\infty}\Bigl|\e  f\Bigr(\Bigl(1 + \sum_{k=1}^R\frac{\tb\zeta_k^2}{1+\tb\sum_{r=1}^{p-1}\xi_{k,r}^2 A'^{-1}_{L(k,r), L(k,r)}}\Bigr)^{-1}\Bigr) -  \e f\bigl(T(A'^{-1}_{NN})\bigr) \Bigr|\\
	&=\Bigl|\E f\Bigl(\Bigl(1 + \sum_{k=1}^R\frac{\tb\zeta_k^2}{1+\tb\sum_{r=1}^{p-1}\xi_{k,r}^2 X_{L(k,r)}}\Bigr)^{-1}\Bigr) - \E f\bigl(T(X)\bigr)\Bigr|\\
	&=\bigl|\E f\bigl(T(X)\bigr)-\E f\bigl(T(X)\bigr)\bigr|=0.
\end{align*}
Since this holds for all Lipschitz $f,$ we conclude that $T(X) \stackrel{d}{=}X$ and Lemma \ref{add:lem9} implies that $X$ is the unique solution to the fixed point equation $T(X)\stackrel{d}{=}X$. Finally, since every weakly converging subsequence of $(A_N^{-1})_{NN}$ converges to the same limit, the entire sequence must converge weakly to the fixed point of $T(X)\stackrel{d}{=}X$ and this completes our proof.

\begin{remark}
	\rm Throughout the entire proof, the assumption on the boundedness of $\mathcal{D}$ is only used when we are applying Proposition \ref{add:prop:indep} to ensure that $X_1,\ldots, X_{l(p-1)},X$ are independent. The general case is handled in the last section where a truncation will be employed.
\end{remark}

\section{Proof of Theorems \ref{thm:fixedpoint} and \ref{thm:freeenergy} for bounded $\mathcal{D}$}

\begin{proof}[\bf Proof of Theorem \ref{thm:fixedpoint} (for $\mathcal{D}$ bounded)] Let $n\geq 1$ be fixed. Since the space of probability distributions on $[0,1]^n$ is compact, for any subsequence of $\bigl((A_N^{-1})_{11},\ldots,(A_N^{-1}\bigr)_{nn})_{N\geq 1}$, we can pass to a subsequence $N_l$ such that
	\begin{align*}
		\bigl((A_{N_l}^{-1})_{11},\ldots,(A_{N_l}^{-1}\bigr)_{nn})_{l\geq 1}
	\end{align*}
	converges weakly to some random vector $(X_1,\ldots,X_n).$ From Proposition \ref{add:prop:indep}, we readily see that $X_1,\ldots,X_n$ are i.i.d. copies of some random variable $X.$ Noting that $(A_{N_l}^{-1})_{11}\stackrel{d}{=}(A_{N_l}^{-1})_{N_lN_l}$, Proposition \ref{add:prop3} ensures that the distribution of $X$ is indeed the unique fixed point of $T.$ In other words, any convergent subsequence of $\bigl((A_N^{-1})_{11},\ldots,(A_N^{-1}\bigr)_{nn})_{N\geq 1}$ has the same limit $(X_1,\ldots,X_n).$ This implies the convergence of $\bigl((A_N^{-1})_{11},\ldots,(A_N^{-1}\bigr)_{nn})_{N\geq 1}$ with the desired limit.
\end{proof}

To establish the proof of Theorem \ref{thm:freeenergy}, we need a technical lemma.

\begin{lemma}\label{add:lem10}
	Let $\lambda >0$. Suppose that $M \sim \mbox{\rm Poisson}(\lambda)$ and $L|M \sim \mbox{\rm Unif}(\{0,\ldots, M\})$ and that $U \sim \mbox{\rm Unif}([0,1])$ and $L' |U\sim \mbox{\rm Poisson}(\lambda U)$. Then $L \stackrel{d}{=}L'$.
\end{lemma}

\begin{proof}
	For any integer $l \geq 0$, we have \begin{align*}
		\P(L= l) & = \sum_{k=l}^\infty \P(L = l | M = k)\P(M=k) = \sum_{k=l}^\infty\frac{1}{k+1}\frac{\lambda^k}{k!} e^{-\lambda} = \frac{e^{-\lambda}}{\lambda}\sum_{k=l+1}^\infty \frac{\lambda^k}{k!},\\
		\P(L' = l) & = \int_0^1 \P(L' = l | U = u)du = \int_0^1 \frac{(\lambda u)^l}{l!} e^{-\lambda u} du.
	\end{align*}
	Here, one can match these two quantities directly by using integration by parts to the last integral for $l$ many times.
	
\end{proof}

\begin{proof}[\bf Proof of Theorem \ref{thm:freeenergy} (for $\mathcal{D}$ bounded)]
	For any $x\in (0,1]$, denote by $X(x)$ the random variable associated to $\mu(\alpha x)$, the fixed point of $T$ associated to the Poisson rate $\alpha x.$
	Recall the identity \eqref{add:eq1} for $F_N.$ From Lemma \ref{lem0}, Theorem \ref{thm:concentration}, and Proposition \ref{add:prop3}, we readily see that
	\begin{align*}
		\limsup_{N\to\infty}	\E \Bigl|\frac{1}{N}\sum_{i,j=1}^NA_{ij}^{-1} -\e X(1)\Bigr| = 0
	\end{align*}
	and
	\begin{align*}
		\limsup_{N\to\infty}\E\Bigl|\frac{1}{N}\log \det A-\frac{1}{N}\e \log \det A\Bigr|=0.
	\end{align*}
	For the remainder of the proof, we handle $$\lim_{N\to\infty}\frac{1}{N}\e\log \det A.$$ Let $S_0 =I$ and for $l\geq 1$, define \begin{align*}
		S_l = I + \tb\sum_{k \leq l}v_kv_k^T.
	\end{align*}
	Thus, $A = S_M$. Write 
	\begin{align*}
		\frac{1}{N}\E  \log \det A & = \frac{1}{N}\E \sum_{l=1}^M \log \frac{\det S_l}{\det S_{l-1}}= \frac{1}{N}\E \sum_{l=1}^M \log \frac{\det \bigl(S_{l-1} + \tb v_l v_l^T\bigr)}{\det S_{l-1}}   \\
		& = \frac{1}{N}\E \sum_{l=1}^M \log \bigl(1 + \tb v_l^T S_{l-1}^{-1}v_l\bigr),
	\end{align*}
	where the last equality used the matrix-determinant lemma. Let $v$ be an $N$-dimensional column vector whose first $p$ entries are $g_1,\ldots,g_p\stackrel{i.i.d.}{\sim}\mathcal{D}$ and the rest are all zero.  Assume that $v$ is independent of all other randomness. We continue to write the last term in the previous display as
	\begin{align*}
		& \frac{1}{N}\sum_{m=0}^\infty \P (M = m)\sum_{l=0}^{m-1}\E \log\bigl(1 + \tb v_{l}^T S_l^{-1}v_l\bigr)\notag\\
		& = \frac{1}{N}\sum_{m=0}^\infty \P (M = m)\sum_{l=0}^{m}\E \log\bigl(1 + \tb v_{l}^T S_l^{-1}v_l\bigr)+ O\Bigl(\frac{1}{N}\Bigr)\\
		& = \frac{1}{N}\sum_{m=0}^\infty \P (M = m)\sum_{l=0}^{m}\E \log\bigl(1 + \tb v^T S_l^{-1}v\bigr)+ O\Bigl(\frac{1}{N}\Bigr)\\
		& = \E \Bigl[\frac{M+1}{N}\E\Bigl[\log\bigl(1 + \tb v^TS_{L}^{-1}v\bigr)\Big| M\Bigr]\Bigr] + O\Bigl(\frac{1}{N}\Bigr), 
	\end{align*}
	where the first equality holds since $\|S_l^{-1}\| \leq 1$ for all $l$, the second equality used the fact that $v_l$ is independent of $S_l$ and the symmetry among the sites of $v_l$, and in the third equality, $L$ depends only on $M$ with the conditional law $L|M\sim \mbox{Unif}(\{0, \ldots, M\})$. 
	
	Next,  From the Cauchy-Schwarz and Jensen inequalities, $M \sim \text{Poisson}(\alpha N)$, and the bound $\|B^{-1}_{L}\|\leq 1$, we have \begin{align*}
		& \E \Bigl[\Bigl|\frac{M+1}{N} - \alpha\Bigr|\E\Bigl[\log(1 + \tb v^TS_{L}^{-1}v)\Big| M\Bigr]\Bigr]\\
		& \leq \Bigl(\E \Bigl|\frac{M+1}{N} - \alpha\Bigr|^2\Bigr)^{1/2} \Bigl(\E \log^2\bigl(1 + \tb v^TS_{L}^{-1}v\bigr)\Bigr)^{1/2} \leq \frac{K'}{N}.
	\end{align*}
	With the help of Lemma \ref{add:lem10} and the fact that $(S_l)_{l\geq 0}$ is independent of $M$,  we arrived at
	\begin{align*}
		\frac{1}{N}\E  \log \det A & = \alpha \E \log\bigl(1+\tb v^TS_{L}^{-1}v\bigr) + O\Bigl(\frac{1}{N}\Bigr),
	\end{align*}
	where we now read $L$ as a random variable with conditional law $L|U\sim\mbox{Poisson}(\alpha U N)$ for some $U\sim \mbox{Unif}([0,1])$ and these are independent of other randomness. From this equation, if we let $L_x$ be an independent Poisson random variable with mean $\alpha xN$, then  
	\begin{align*}
		&\frac{1}{N}\E  \log \det A = \alpha \E\Bigl[\log(1+\tb v^TS^{-1}_{L}v)\Bigr] + O\Bigl(\frac{1}{N}\Bigr)\notag\\
		& = \alpha\int_0^1\E \log(1+\tb v^TS^{-1}_{L_x}v) dx +  O\Bigl(\frac{1}{N}\Bigr)\notag\\
		& = \alpha\int_0^1 \E\log\Bigl(1+\tb\sum_{r=1}^pg_{r}^2 (S^{-1}_{L_x})_{rr} + 4\beta\sum_{1\leq r< s\leq p}g_{r}g_{s}(S^{-1}_{L_x})_{rs}\Bigr)dx + O\Bigl(\frac{1}{N}\Bigr). 
		%\label{eq9}
	\end{align*}
	Notice that for each $x$ the matrix $S_{L_x}$ is the same as the matrix $A$ in distribution, except that 
	$M \sim \mbox{Poisson}(\alpha N)$ has been replaced by $L_x \sim \mbox{Poisson}(\alpha x N)$. We can employ the same argument as \eqref{add:eq-13} to get that $$
	\sup_{x\in [0,1]}\max_{1\leq r<s\leq N}\e\bigl|(S^{-1}_{L_x})_{rs}\bigr|^2=O\Bigl(\frac{1}{N}\Bigr).$$ Consequence, using the inequality \begin{align*}
		|\log(1+u) - \log(1+v)| \leq |u-v|,\,\,u,v>0
	\end{align*} 
	yields
	\begin{align*}
		\frac{1}{N}\E  \log \det A =\int_0^1\E\log\Bigl(1+\tb\sum_{r=1}^pg_{r}^2 (S^{-1}_{L_x})_{rr} \Bigr)dx+O\Bigl(\frac{1}{\sqrt{N}}\Bigr).
	\end{align*}
	Together with Theorem \ref{thm:fixedpoint} (the only place where the boundedness of $\mathcal{D}$ is needed), which although stated for $A$, continue to hold for $S_{L_x}$ under the replacement of $\alpha$ by $\alpha x$, we have \begin{align*}
		\lim_{N\to \infty} \E \frac{1}{N} \log \det A & = \alpha\int_0^1 \E\log\Bigl(1+\tb\sum_{r=1}^pg_{r}^2 X_r(x)\Bigr)dx
	\end{align*}
	concluding our proof.
\end{proof}

\section{Proof of Theorems \ref{thm:fixedpoint} and \ref{thm:freeenergy} for general $\mathcal{D}$}

Throughout this section, we assume that $\mathcal{D}$ has finite second moment. Recall that $(g_{k,i})_{k\geq 1, 1\leq i\leq N}$ are i.i.d. copies of $\mathcal{D}$. Also recall the $v_k$'s from the definition of $A$ in \eqref{def_A}.
For a truncation level $c >0$, define \begin{align*}
	g_{k,i}^c& = g_{k,i}\mathbb{I}(|g_{k,i}| \leq c)
\end{align*}
and define the vectors $  v_k^c$ by replacing $(g_{k,i})_{1\leq i\leq N}$ in $v_k$ with $(g_{k,i}^c)_{1\leq i\leq N}$. Set \begin{align*}
	A^c = I + \tb \sum_{k \leq M} v_k^c{v_k^c}^T.
\end{align*} Likewise, $\xi_{k,i}^c$ and $\zeta_k^c$ are defined accordingly from  $\xi_{k,i}$ and $\zeta_k$ with the same truncation level $c$. Define $T_c$ as the operator $T$ in \eqref{conv:spinvariance:eq1} with the replacement of $\xi_{k,i}$ and $\zeta_k$ by ${\xi}_{k,i}^c$ and $ {\zeta}_k^c$.  Denote by $ \mu_c$ the unique fixed point of $T_c$. 

\begin{lemma}\label{prop:weakconv}
	As $c\to\infty,$ $\mu_c$ converges weakly to $\mu_\infty$, the unique fixed point of  the operator $T$.
\end{lemma}

\begin{proof}
	Let $\mu_0$ be the weak limit of some convergent subsequence $(\mu_{c_l})_{l \geq 1}$ of the family of tight probability measures $(\mu_c)_{c>0}$. We claim that $\mu_0$ is a fixed point of $T.$ If this holds, then from Lemma~\ref{add:lem9}, $\mu_0$ must be the unique fixed point of $T$, which implies that every weakly convergent subsequence of $(\mu_c)_{c>0}$ shares the same limit $\mu_0$ and this concludes that $(\mu_c)_{c>0}$ converges to $\mu_0$ as $c\to\infty.$
	
	We now turn to the proof of our claim. To ease our notation, without loss of generality, we assume that $(\mu_{c})_{c>0}$ converges to $\mu_0$. Let $d_L$ be the L\'evy metric on $\mathcal{P}([0,1])$.
	We write
	\begin{align*}
		d_L(T(\mu_0),\mu_0)&\leq d_L(T(\mu_0),T_c(\mu_c))+d_L(T_c(\mu_c),\mu_0).
	\end{align*} 
	Our proof will be completed if the two terms on the right-hand side vanishes as $c$ tends to infinity. The second term obviously converges to zero since $T_c(\mu_c)=\mu_c$ converges to $\mu_0$ weakly. To handle the first term, let $(X_{k,r}^c)_{k\geq 1,1\leq r\leq p-1}\stackrel{i.i.d.}{\sim}\mu_c$ and $(X_{k,r})_{k\geq 1,1\leq r\leq p-1}\stackrel{i.i.d.}{\sim}\mu_0.$ Assume that these are independent of each other and everything else. Now since $\mu_c$ converges to $\mu_0$ weakly and $(X_{k,r}^c)_{k\geq 1,1\leq r\leq p-1},(\xi_{k,r}^c)_{k\geq 1,1\leq r\leq p-1},	(\zeta_{k}^c)_{k\geq 1}$ are all independent of each other, it follows that for any $l\geq 0,$ we have the following joint convergence
	\begin{align*}
		&\Bigl((X_{k,r}^c)_{1\leq k\leq l,1\leq r\leq p-1},(\xi_{k,r}^c)_{1\leq k\leq l,1\leq r\leq p-1},	(\zeta_{k}^c)_{1\leq k\leq l}\Bigr)\\
		&\Rightarrow \Bigl((X_{k,r})_{1\leq k\leq l,1\leq r\leq p-1},(\xi_{k,r})_{1\leq k\leq l,1\leq r\leq p-1},	(\zeta_{k})_{1\leq k\leq l}\Bigr).
	\end{align*}
	Consequently, for any $l\geq 0,$
	\begin{align*}
		U_l^c&:=\Bigl(1+\sum_{k=1}^l \frac{2\beta (\zeta_k^c)^2}{1+2\beta\sum_{r=1}^{p-1}X_{k,r}^c(\xi_{k,r}^c)^2}\Bigr)^{-1}\\
		&\xRightarrow{c\to\infty}	\Bigl(1+\sum_{k=1}^l\frac{2\beta \zeta_k^2}{1+2\beta\sum_{r=1}^{p-1}X_{k,r}\xi_{k,r}^2}\Bigr)^{-1}=:U_l.
	\end{align*}
	Consequently,  if $Y_c\sim T_c(\mu_c)$ and $Y_0\sim T(\mu_0),$ then for any bounded continuous function $f,$
	\begin{align*}
		\e f(Y_c)&=\sum_{l=0}^\infty \e f(U_l^c) \P(R=l)\to \sum_{l=0}^\infty \e f(U_l) \P(R=l)=\e f(Y_0),
	\end{align*}
	where we used the dominated convergence theorem. As a result, $T_c(\mu_c)\Rightarrow T(\mu_0)$, which implies that $d_L(T(\mu_0),T_c(\mu_c))\to 0$ and this completes our proof. 
\end{proof}

We are ready to establish Theorem \ref{thm:freeenergy} for the general case.

\begin{proof}[\bf Proof of Theorem \ref{thm:freeenergy}] 
	Let  $F_N^c$ be the free energy of our model associated to the truncated disorders $(g_{k, i}^c)_{k \ge 1, i \ge 1}.$
	Our goal is to show
	\begin{align}\label{eq:ctr3}
		\lim_{c\to\infty}\limsup_{N\to\infty}\frac{1}{N}\bigl|\e F_N-\e F_N^c\bigr|=0.
	\end{align}
	Once we prove \eqref{eq:ctr3}, Theorem~\ref{thm:freeenergy} follows from Theorem \ref{thm:concentration}, Theorem \ref{thm:freeenergy} for the bounded case, and Lemma \ref{prop:weakconv} with an application of the continuous mapping theorem of weak convergence.
	
	To show \eqref{eq:ctr3}, from  \eqref{add:eq1} and \eqref{lem0:eq1}, we have 
	\begin{align*}
		\e F_N&=\frac{h^2}{2}\frac{\e\mbox{tr}\bigl(A^{-1}\bigr)}{N} +\frac{\e \log\det A}{2N},\\
		\e F_N^c&=\frac{h^2}{2}\frac{\e\mbox{tr}\bigl((A^c)^{-1}\bigr)}{N} +\frac{\e\log \det A^c}{2N}.
	\end{align*}
	Let $f_c(x) := \log(1+x \wedge c)$  for $x \ge 0$. Applying \eqref{eq:ctr1} to the original and truncated disorders, 
	we obtain
	\begin{align} \label{eq:ctr4}
		\limsup_N  \Bigl|	\frac{1}{N} \e \log   \det A-\frac{1}{N} \e \mbox{tr} f_c (A-I)\Bigr|& \le \frac{K}{\sqrt{c}},  \\\limsup_N  \Bigl|	\frac{1}{N} \e \log   \det A^c-\frac{1}{N} \e \mbox{tr} f_c (A^c-I)\Bigr|&\le \frac{K}{\sqrt{c}} \label{eq:ctr5}
	\end{align} 
	for some constant $K$ that does not depend on $c$ and $N$. Next, let $\mu_{A-I}$ and $\mu_{A^c-I}$ be the empirical spectral measures of $A-I$ and $A^c-I$, respectively. 	From \eqref{empiricalsd}, we obtain that
	\begin{align}
		\Bigl|  \frac{1}{N} \e \mbox{tr} f_c (A-I) -  \frac{1}{N} \e \mbox{tr} f_c (A^c-I)  \Bigr| &=  \Bigl| \e \int_0^\infty  f_c d\bigl(\mu_{A-I}-\mu_{A^c-I}\bigr)\Bigr| \nonumber \\
		&\le \frac{ \| f_c \|_{\mbox{\tiny BV}} \e [\mathrm{rank}(A  - A^c)]}{N} \nonumber \\
		&\le  \frac{ 2\log (1+c) \e \sum_{k=1}^M \mathbb{I}( v_k \ne v_k^c )}{N} \nonumber \\
		&\le 2\alpha p   \log (1+c)   \P( |g_{1, 1}| > c).\label{eq:ctr6}
	\end{align} 
	As a consequence of \eqref{eq:ctr4}, \eqref{eq:ctr5} and \eqref{eq:ctr6},  we have the following bound
	\begin{align*}
		\limsup_N  \Bigl|	\frac{1}{N} \e \log   \det A- 	\frac{1}{N} \e \log   \det A^c  \Bigr| &\le  \frac{2K}{\sqrt{c}} + 2\alpha p   \log (1+c)   \P( |g_{1, 1}| > c)\\
		&\le \frac{2K}{\sqrt{c}} +2 \alpha p   \log (1+c)   \frac{ \e |g_{1, 1}| }{c} ,
	\end{align*}
	which yields that
	\begin{equation}\label{eq:ctr8}
		\lim_{c \to \infty}  \limsup_N  \Bigl|	\frac{1}{N} \e \log   \det A- 	\frac{1}{N} \e \log   \det A^c  \Bigr|  = 0. 
	\end{equation}
	Next we define  $f(x) = 1/(1+x)$ for $x \ge 0$ and argue similarly as above to conclude that 
	\begin{align*}
		\Bigl|  \frac{1}{N} \e\mbox{tr}\bigl(A^{-1}\bigr)  -   \frac{1}{N} \e\mbox{tr}\bigl((A^c)^{-1}\bigr)  \Bigr| &=  \Bigl| \e \int_0^\infty  f d\bigl(\mu_{A-I}-\mu_{A^c-I}\bigr)\Bigr| \le 2\alpha p    \P( |g_{1, 1}| > c),  
	\end{align*} 
	which implies that 
	\begin{equation}\label{eq:ctr9}
		\lim_{c \to \infty}  \limsup_N  \Bigl|  \frac{1}{N} \e\mbox{tr}\bigl(A^{-1}\bigr)  -   \frac{1}{N} \e\mbox{tr}\bigl((A^c)^{-1}\bigr)  \Bigr| = 0. 
	\end{equation}
	Now  \eqref{eq:ctr3} follows immediately from \eqref{eq:ctr8} and \eqref{eq:ctr9}.
\end{proof}

The proof of Theorem \ref{thm:fixedpoint} is based on  the following lemma.

\begin{lemma}\label{sec8:lem1}
	We have that
	\begin{align*}
		\lim_{c\to\infty}	\limsup_{N\to\infty}\e\bigl|A_{11}^{-1}-\bigl((A^c)^{-1}\bigr)_{11}\bigr|=0.
	\end{align*}
\end{lemma}
\begin{proof}
	Define 
	\[ \mathcal{A}=  I+2\beta\sum_{k=1}^M \epsilon_k v_kv_k^T = I+2\beta\sum_{k=1}^M \epsilon_k (v_k^c) (v_k^c)^T,\]
	where $\epsilon_k$'s are Bernoulli variables given by 
	\[ \epsilon_k  =  \left\{ \begin{array}{ll} 1, &\mbox{if $|g_{k,I(k,r)}|\leq c$ for all $1\leq r\leq p$,} \\
		\\
		0, & \text{otherwise}.
	\end{array}   \right. \]
	Note that $\epsilon_k = 1$ implies that $v_k = v_k^c$, which justifies the second equality in the definition of $\mathcal{A}$. We have $A \ge\mathcal{A}$ and $A^c \ge \mathcal{A}$. Therefore, $ A^{-1}_{11} \le \mathcal{A}^{-1}_{11}$ and $ ((A^c)^{-1})_{11} \le \mathcal{A}_{11}^{-1}$. Consequently,
	\begin{align*}
		\bigl|A_{11}^{-1}- ((A^c)^{-1})_{11}\bigr| \le \bigl ( \mathcal{A}_{11}^{-1} -  A_{11}^{-1} \bigr) + \bigl( \mathcal{A}_{11}^{-1}-  ((A^c)^{-1})_{11}\bigr).
	\end{align*} 
	Taking expectation and using the symmetry of the spin coordinates, we obtain
	\begin{align*}
		\e \bigl|A_{11}^{-1}- ((A^c)^{-1})_{11}\bigr|  &\le \Bigl (  \frac{1}{N} \e\mbox{tr}\mathcal{A}^{-1} -   \frac{1}{N} \e\mbox{tr}\bigl(A^{-1}\bigr)  \Bigr) + \Bigl (  \frac{1}{N} \e\mbox{tr}\mathcal{A}^{-1}   -   \frac{1}{N} \e\mbox{tr}\bigl((A^c)^{-1}\bigr)  \Bigr) \\
		&=  \int_0^\infty  f d\bigl(\mu_{\mathcal{A}-I}-\mu_{A-I}\bigr) + \int_0^\infty  f d\bigl(\mu_{\mathcal{A}-I}-\mu_{A^c-I}\bigr), 
	\end{align*}
	where $f(x)= 1/(1+x)$ for $x \ge 0$ and $\mu_{A-I}, \mu_{A^c-I}$ and $\mu_{\mathcal{A}-I}$ are the empirical spectral measures of $A-I, A^c  - I$ and $\mathcal{A} -I$, respectively. Observe that both $ \mathrm{rank}( A - \mathcal{A})$  and $ \mathrm{rank}(A^c- \mathcal{A})$ are bounded above by  
	$p\sum_{k=1}^M \mathbb{I} ( \epsilon_k = 0).$
	Consequently, 
	\begin{align*}
		\e  \int_0^\infty  f d\bigl(\mu_{\mathcal{A}-I}-\mu_{A-I}\bigr) &\le \frac{ \| f_c \|_{\mbox{\tiny BV}} \e [ \mathrm{rank}(A  - \mathcal{A})]}{N} \\
		&\le   \frac{ p\e\sum_{k=1}^M \mathbb{I} ( \epsilon_k = 0)}{N}\le  \alpha p^2     \P( |g_{1, 1}| > c).
	\end{align*} 
	The same bound also holds for $\e  \int_0^\infty  f d\bigl(\mu_{\mathcal{A}-I}-\mu_{A^c-I}\bigr)$. Therefore, we arrive at the bound 
	\[ \e \bigl|A_{11}^{-1}- ((A^c)^{-1})_{11}\bigr|  \le 2 \alpha p^2     \P( |g_{1, 1}| > c)\]
	and the assertion follows. 
\end{proof}
\begin{proof}[\bf Proof of Theorem \ref{thm:fixedpoint}] Let $n\geq 1$ be fixed and $f$ be a Lipschitz function on $\mathbb{R}^n.$
For each fixed $c>0,$ from Theorem \ref{thm:fixedpoint} for bounded $\mathcal{D}$ case, we have
\begin{align*}
	\lim_{N\to\infty}\e f\bigl((A^c)^{-1})_{11},\ldots, (A^c)^{-1})_{nn}\bigr)=\e f\bigl(X_{1}^c,\ldots,X_n^c\bigr),
\end{align*}
where $X_1^c,\ldots,X_n^c$ are i.i.d. copies of $\mu_c.$ It follows that from Lemma \ref{prop:weakconv},
\begin{align*}
	\lim_{c\to\infty}\lim_{N\to\infty}\e f\bigl((A^c)^{-1})_{11},\ldots, (A^c)^{-1})_{nn}\bigr)=\lim_{c\to\infty}\e f\bigl(X_{1}^c,\ldots,X_n^c\bigr)=\e f\bigl(X_1,\ldots,X_n\bigr)
\end{align*}
for $X_1,\ldots,X_n$ i.i.d. copies of $\mu_\infty$, the unique fixed point of the operator $T.$ Finally, our proof is completed by applying Lemma~\ref{sec8:lem1},
\begin{align*}
	\lim_{c\to\infty}\limsup_{N\to\infty}\Bigl|	\e f\bigl(A_{11}^{-1},\ldots,A_{nn}^{-1}\bigr)-\e f\bigl((A^c)^{-1})_{11},\ldots, (A^c)^{-1})_{nn}\bigr)\Bigr|=0.
\end{align*}
\end{proof}
\bibliographystyle{imsart-number} % Style BST file (imsart-number.bst or imsart-nameyear.bst)
\bibliography{ref}       % Bibliography file (usually '*.bib')

%% or include bibliography directly:
%\begin{thebibliography}{}
%\bibitem{b1}
%\end{thebibliography}

\end{document}